\documentclass[11pt]{amsart}
\usepackage{amsfonts}
\usepackage{bbm}
\usepackage{mathrsfs,txfonts}
\usepackage{color}
\usepackage{amssymb,amsmath,amsfonts,amsthm}
\usepackage[colorlinks=true, pdfstartview=FitV, linkcolor=blue, citecolor=blue, urlcolor=blue]{hyperref}
\numberwithin{equation}{section}

\newtheorem{thm}{Theorem}[section]
\newtheorem{lem}{Lemma}[section]
\newtheorem{cor}{Corollary}[section]

\newtheorem{prop}[thm]{Proposition}

\theoremstyle{remark}
\newtheorem{rem}[thm]{Remark}

\renewcommand{\ddot}[1]{{#1}'}
\renewcommand{\dot}[1]{{#1}}
\newcommand{\pr}[1]{\left( #1 \right)}
\newcommand{\tn}[1]{\textnormal{#1}}
\newcommand{\abs}[1]{\left| #1\right|}
\newcommand{\n}[2]{{\left\| #1 \right\|}_{#2}}
\newcommand{\f}[2]{\frac{#1}{#2}}
\newcommand{\lan}[1]{\left\langle #1\right\rangle}
\newcommand{\wh}[1]{\widehat{#1}}
\newcommand{\wt}[1]{\,\widetilde{#1}\,}
\newcommand{\lp}{L^{\ddot{p}}_{\ddot{{a}}}  L^{\dot{p}}_{\dot{a}}}
\newcommand{\lpa}{L^{\ddot{p}_1}_{\ddot{a}_1}  L^{\dot{p}_1}_{\dot{a_1}}}
\newcommand{\lpb}{L^{\ddot{p}_2}_{\ddot{a}_2}  L^{\dot{p}_2}_{\dot{a}_2}}
\newcommand{\al}{\alpha}

\newcommand{\ga}{\gamma}

\newcommand{\de}{\delta}
\newcommand{\De}{\Delta}

\newcommand{\ve}{\varepsilon}

\newcommand{\si}{\sigma}

\newcommand{\na}{\nabla}

\newcommand{\rn}{\mathbf{R}^n}

\newcommand{\bn}{\mathbf{N}}

\newcommand{\cs}{\mathcal S}

\newcommand{\cf}{\mathcal F}

\newcommand{\p}{\partial}

\newcommand{\ds}{\displaystyle}

\theoremstyle{definition}

\def\rr{\mathbf{R}}
\def\nn{\mathbf{N}}

\def\zz{\mathbf{Z}}

\def\mm{\mathbf{m}}

\def\supp{\text{supp}}

\def\({\left(}
\def\){\right)}

\def\S{\mathcal{S}}
\def\A{\mathbf{A}}

\allowdisplaybreaks

\numberwithin{equation}{section}

\begin{document}

\title[The Kato-Ponce Inequality with Polynomial Weights]{The Kato-Ponce Inequality with Polynomial Weights}

\author[S. Oh and X. Wu]
{Seungly Oh and Xinfeng Wu}
\bigskip

\address{Department of Mathematics\\
Western New England University\\
Springfield, MA-01119, U.S.A.}
\email{seungly.oh@wne.edu}

\address{Department of Mathematics\\
China University of Mining \& Technology(Beijing)\\
Beijing, 100083, P. R. China}
\email{wuxf@cumtb.edu.cn}

\thanks{The second author supported by NNSF of China (Nos. 11671397, 12071473).
 }


\begin{abstract}{We consider various versions of fractional Leibniz rules (also known as Kato-Ponce inequalities) with polynomial weights $\langle x\rangle^a = (1+|x|^2)^{a/2}$ for $a\ge 0$.  We show that the weighted Kato-Ponce estimate with the inhomogeneous Bessel potential $J^s = (1- \De)^{{s}/{2}}$ holds for the full range of bilinear Lebesgue exponents, for all polynomial weights, and for the sharp range of the degree $s$. This result, in particular, demonstrates that neither the classical Muckenhoupt weight condition nor the more general multilinear weight condition is required for the weighted Kato-Ponce inequality. We also consider a few other variants such as commutator and mixed norm estimates, and analogous conclusions are derived.  Our results contain strong-type inequalities for both $L^1$ and $L^\infty$ endpoints, which extend  several existing results.
}\end{abstract}

\maketitle
\tableofcontents
\section{Introduction}

Boundedness properties of fractional differential operators acting on a product of functions have a wide range of applications, not only in functional analysis, but in the analysis of nonlinear partial differential equations.  A typical statement in this context concerns how the (fractional) differential operator acting on the product can be controlled using the ``derivatives'' of individual functions.  Such properties can be used to manage nonlinear interaction terms in Sobolev spaces for the analysis of partial differential equations.   In literature, such estimates are known as fractional Leibniz rules, since they resemble the structure of the product rule for classical derivatives.

Motivated by a question posed by Kato in \cite{Kato}, Kato and Ponce showed in~\cite{KP} that any pair of Schwartz functions $f$, $g$ satisfies
\[
\n{J^{{s}} (fg) - f \, J^s g}{L^p(\rr^n)} \leq C\pr{  \n{\na f}{L^\infty(\rr^n)} \n{J^{{s-1}} g}{L^p(\rr^n)} + \n{J^{{s}} f}{L^p(\rr^n) } \n{g}{L^\infty(\rr^n)}},
\]
where $1<p<\infty$, $s>0$, and $C$ is a constant depending only on $n$, $p$ and $s$, which is known as the Kato-Ponce commutator estimate. The authors  applied this estimate in the analysis of Euler and Navier-Stokes equations.  If $J^s$ above is replaced by any classical differential operator, we can see that $f J^s g$ would cancel the highest-order derivative on $g$ within the expansion of $J^s (fg)$ according to the classical Leibniz rule, which would leave the highest-order (with respect to $g$) term to be  $\p_j f J^{s-1} g$.  This inequality demonstrates that a similar property  holds for the fractional differential operator $J^s$.  Many variants of such estimate have been studied in the literature (see \cite{BO, BL, M, MS13a, Tao} and references therein).

Another classical result in this regard is given in \cite{KPV}, where Kenig, Ponce and Vega proved a similar property for the homogeneous differential operator $D^s := (-\De)^{\f{s}{2}}$ and applied the estimate in the analysis of generalized Korteweg-de Vries equation.  The authors showed that any pair of Schwartz functions $f$ and $g$ satisfies
    \[
    \n{D^s(fg) - f D^s g - g D^s f}{L^p(\rr^n)} \leq C \n{D^{s_1} f}{L^{p_1}(\rr^n)} \n{D^{s_2} g}{L^{p_2}(\rr^n)}
    \]
for $s_1, s_2, s\in (0,1)$  and $p_1,p_2, p \in (1,\infty)$ satisfying $s_1 + s_2 = s$, and $\f{1}{p_1} + \f{1}{p_2} = \f{1}{p}$.    While this statement is restricted to $s\in (0,1)$, generalizations exist in the literature which allows for $s>1$ by further complicating the structure of this inequality; see {\it e.g.,}  \cite{FGO,L}.

The following variants are known as the Kato-Ponce inequality (or fractional Leibniz rule)
\begin{align}
\n{D^s (fg)}{L^p(\rr^n)} &\lesssim  \n{D^s f}{L^{p_1}(\rr^n)} \n{g}{L^{p_2}(\rr^n)} + \n{f}{L^{p_1}(\rr^n)}\n{D^s g}{L^{p_2}(\rr^n)}, \label{eq:kp}\\
\n{J^s (fg)}{L^p(\rr^n)} &\lesssim  \n{J^s f}{L^{p_1}(\rr^n)} \n{g}{L^{p_2}(\rr^n)} + \n{f}{L^{p_1}(\rr^n)}\n{J^s g}{L^{p_2}(\rr^n)}, \label{eq:kp2}
\end{align}
where $1<p_1,p_2,p<\infty$ satisfy the natural H\"{o}lder exponent condition  $\frac{1}{p_1}+ \frac{1}{p_2}=\frac{1}{p}$.
Christ and Weinstein \cite{CW} proved \eqref{eq:kp} for $s\in (0,1)$ and applied it for the analysis of generalized Korteweg-de Vries equation, and Gulisashvili and Kon \cite{GK} showed \eqref{eq:kp} and \eqref{eq:kp2} for $s>0$ and applied them in the analysis of Schr\"odinger semigroups. Both results restricted the target index within the non-endpoint Banach range, $1<p<\infty.$

Since these inequalities were frequently used in applications, they became research subjects of independent interest in the field of multilinear harmonic analysis; see {\it e.g.,}  \cite{Grafakos2,G12}.  In this field, the $L^{p_1}\times L^{p_2}\to L^p$ boundedness of bilinear operators often extends to the quasi-Banach regime $p < 1$ as long as $p_1,p_2\in (1,\infty)$; see \cite{GT,KS,LT1,LT2}.  In \cite{BMMN}, Bernicot, Maldonado, Moen, and Naibo extended \eqref{eq:kp} to allow for $p<1$, under an additional assumption that $s> n$.  This result was further extended by Muscalu and Schlag \cite{MS13a} and also by Grafakos and the first author of this manuscript \cite{GO} to allow for a wider range $s>\max\left\{n\pr{\f{1}{p}-1},0\right\}$ or $s\in 2\nn$.  Furthermore, the sharpness of this range was demonstrated by counterexamples in \cite{GO,MS13a} when $s\leq \max\left\{n\pr{\f{1}{p}-1},0\right\}$ and $s\not\in 2\nn$.

The endpoint case $p_1=p_2=p=\infty$ was first considered by Grafakos, Maldonado and Naibo \cite{GMN}, where the authors were able to show the corresponding estimates with the target space $L^{\infty}$ replaced by BMO.  This is a natural conclusion in context of existing results in multilinear harmonic analysis, since bilinear Coifman-Meyer multipliers (more generally, bilinear Calder\'on-Zygmund operators) are known \cite{GT} to be bounded from $L^\infty \times L^\infty$ to $BMO$ (rather than to $L^\infty$). Moreover, standard methods used in the literature to derive the boundedness properties of bilinear operators, such as the vector-valued maximal inequality, the square-function estimate and Coifman-Meyer multiplier theorem, fail to be valid at   $L^\infty$ endpoint (that is, $L^\infty\times L^\infty \to L^\infty$) as well as $L^1$ endpoints  (more specifically, $L^1\times L^{p_2} \to L^p$).  Instead, these techniques lead to weaker results at these endpoints, namely, $L^1\times L^{p_2} \to L^{p,\infty}$ and $L^\infty\times L^\infty \to BMO$.

In this context of multilinear harmonic analysis, it was surprising when Bourgain and Li \cite{BL} showed that the strong-type estimates \eqref{eq:kp} and \eqref{eq:kp2} hold at the $L^\infty$ endpoint.  Their proof used a low-to-high frequency exchange in order to achieve summability at this endpoint.  Adaptation of this method by the authors of the present manuscript \cite{OW} produced strong-type estimates at the $L^1$ endpoints, completing the validity of these inequalities in the full quasi-Banach regime, $\f{1}{2}\leq p \leq \infty$ and $1\leq p_1, p_2 \leq \infty$.  This indicates that the Kato-Ponce inequalities should be distinguished from other  multilinear estimates which fail to satisfy strong-type $L^1$ and $L^\infty$ bounds.

This manuscript will demonstrate that the validity of various Kato-Ponce inequalities also exceeds conventional thresholds of weighted multilinear estimates.  The weighted Lebesgue norm $L^p_w$ (also denoted $L^p(w)$) is defined by
\[
\n{f}{L^p_w} := \pr{\int_{\rr^n} |f(x)|^p w(x)\, dx}^{\f{1}{p}}
\]
for $0<p < \infty$, while $\n{f}{L^\infty_w} = \n{f}{L^\infty}$.

Often, weighted inequalities in harmonic analysis are only expected to be valid when the weights satisfy the Muckenhoupt weight condition, also known as the $A_p$~condition.  This is a natural assumption since it is well-known \cite{Grafakos2, stein93} that the $w\in A_p$ for $1<p<\infty$ is a necessary and sufficient condition for the $L^p_{w}$ boundedness of the Hilbert transform, as well as the Hardy-Littlewood maximal operator.

The weighted variants of \eqref{eq:kp} and \eqref{eq:kp2} were proved by Naibo and Thomson \cite{NT} in non-endpoint cases, that is, $1<p_1,p_2<\infty$, when the weights satisfied the $A_p$ condition.  More specifically, these results extended \eqref{eq:kp2} to
\begin{align}\label{1.4}
\n{J^s (fg)}{L^p_{w}(\rr^n)} &\lesssim  \n{J^s f}{L^{p_1}_{w_{1}}(\rr^n)} \n{g}{L^{p_2}_{w_{2}}(\rr^n)} + \n{f}{L^{p_1}_{w_{1}}(\rr^n)}\n{J^s g}{L^{p_2}_{w_{2}}(\rr^n)},
\end{align}
for weights $w_1\in A_{p_1}$, $w_2\in A_{p_2}$, $w= w_1^{p/p_1}w_2^{p/p_2}$ and $s>\max\left\{n\pr{\f{\tau(w)}{p}-1},0\right\}$, where $\tau(w) = \inf \{\tau \in (1,\infty): w \in A_\tau\}$.  This problem has been previously investigated by Cruz-Uribe and Naibo \cite{CN}, which will be discussed in the context of our new results in Section~\ref{sec:inhom}.

In comparison, a more general class of multilinear weights (denoted by $\A_{\vec p},\vec{p}=(p_1,\ldots,p_m)$) and multilinear maximal function were introduced by Lerner, Ombrosi, P\'{e}rez, Torres, and Trujillo-Gonz\'{a}lez in \cite{LOPTT}. The $m$-linear maximal function is smaller than the $m$-fold product of the Hardy-Littlewood maximal functions, while the multilinear weights class $\A_{\vec{p}}$ is strictly larger than the $m$-fold tensor product $A_{p_1}\otimes\cdots \otimes A_{p_m}$.  They proved that $(w_1,\ldots, w_m)\in \A_{\vec{p}}$ is a necessary and sufficient condition for the  boundedness of the $m$-linear maximal operator on $L^{p_1}_{w_1}\times \cdots\times L^{p_m}_{w_m}, 1<p_1,\ldots,p_m<\infty$.  The authors also proved similar multilinear weighted inequalities for multilinear singular integrals (in particular, multilinear Coifman-Meyer multipliers) and related commutators.   Since the paraproduct decomposition of the LHS of \eqref{1.4} reduces the estimate to that of bilinear Coifman-Meyer multipliers, it is not difficult to deduce that \eqref{1.4} holds with $(w_1,w_2)\in \A_{p_1,p_2}$ as long as $s$ is sufficiently large (for instance, $s>2n+1$).  However, such argument fails to produce the sharp range of $s$ or the endpoint Lebesgue exponents.

In this article, we will investigate \eqref{1.4} with a specific class of weights, $\lan{x}^a$ for $a\geq 0$, which are known as power weights or polynomial weights.  We prove that \eqref{1.4} holds for all pair of polynomial weights, even beyond the $A_{p_1,p_2}$ range, and for the full range of Lebesgue exponents $\f{1}{2}\leq p \leq \infty$ and $1\le p_1,p_2\le \infty$, where the sharp range of $s$ is also demonstrated.  Our result portrays that there are substantial differences between weighted Kato-Ponce inequalities and typical multilinear weighted estimates in harmonic analysis.  A few of these differences are highlighted as follows:
\begin{itemize}
    \item The strong-type weighted inequality \eqref{1.4} remains true at the $L^1$ and $L^\infty$ endpoints,
    i.e. $L^1 \times L^{p_2} \to L^p$ and $L^\infty\times L^\infty \to L^\infty$; in particular, the strong type $L^{1}_{\lan{x}^{a_1}}\times L^1_{\lan{x}^{a_2}}\to L^{1/2}_{\lan{x}^{\frac{a_1+a_2}{2}}}$ estimate holds
    for all $a_1,a_2\in[0,\infty)$.
    \item Neither the Muckenhoupt $A_{p_1}\otimes A_{p_2}$ condition nor the multilinear $\A_{p_1,p_2}$ condition is  necessary for the weighted Kato-Ponce inequality \eqref{1.4}.
\end{itemize}

Our first main result is stated below.

\begin{thm}\label{main1}
Let  $\frac{1}{2}\le p \le \infty$, $1\le p_1,p_2\le \infty$, and $a, a_1,a_2\ge 0$ satisfy
\begin{equation}\label{eq:cond}
\f{1}{p} = \f{1}{p_1} + \f{1}{p_2}, \qquad \f{a}{p} = \f{a_1}{p_1}+ \f{a_2}{p_2},
\end{equation}
Then, if $s>\max\left\{0,n\pr{\f{1}{p} - 1}\right\}$ or $s\in 2\nn$,
\begin{align}\label{eq:1.3}
\n{J^s (fg)}{L^p_{ \langle x\rangle^a
}(\rr^n)} &\lesssim  \n{J^s f}{L^{p_1}_{\langle x \rangle^{a_1}}(\rr^n)} \n{g}{L^{p_2}_{\langle x \rangle^{a_2}}(\rr^n)} + \n{f}{L^{p_1}_{\langle x \rangle^{a_1}}(\rr^n)}\n{J^s g}{L^{p_2}_{\langle x \rangle^{a_2}}(\rr^n)}. 
\end{align}
Moreover, the range of $s$ above is optimal. More precisely, if $s\le \max\left\{0,n\pr{\f{1}{p} - 1}\right\}$ and $s\notin 2\nn,$ then \eqref{eq:1.3} fails in general.
\end{thm}

We note that Theorem~\ref{main1} considers
the full range of polynomial weights: $w_1=\langle x\rangle^{a_1}, w_2=\langle x\rangle^{a_2}$ for $a_1,a_2\in [0,\infty)$.
Since, \cite{Grafakos}, $\lan{x}^a \in A_p$ if and only if $-n< a < n(p-1)$ , \cite[Theorem 3.6]{LOPTT} and direct computations show that
\[
(\langle x\rangle^{a_1},\langle x\rangle^{a_2}) \in \A_{p_1,p_2}
   \iff (\langle x\rangle^{a_1},\langle x\rangle^{a_2})\in A_{p_1}\otimes A_{p_2} \iff a_i\in (-n,n(p_i-1)), i=1,2.
\]
In particular, these weights do not satisfy the standard Muckenhoupt weight condition, $(w_1,w_2)\in A_{p_1}\otimes A_{p_2}$ or the multilinear condition $(w_1,w_2)\in \A_{p_1,p_2}$ when either $a_1$ or $a_2$ is sufficiently large.
As mentioned above, Theorem~\ref{main1} portrays that the $\A_{p_1,p_2}$~condition is not necessary for the weighted Kato-Ponce inequality \eqref{1.4}.  Even when the weights satisfy the Muckenhoupt condition, Theorem \ref{main1} extends the result in \cite{NT} to allow for $p_1\in \{1,\infty\}$ or $p_2\in \{1,\infty\}$, and determines the optimal range of $s$, at least for the polynomial weights. When $p=p_1=p_2=\infty$, Theorem~\ref{main1} is included in \cite[Theorem 1.1]{BL} since $L^\infty_w = L^\infty$.

We will also establish the weighted Kato-Ponce commutator estimate as given below:
\begin{thm}\label{main-comm}
Let  $\frac{1}{2}\le p \le \infty$, $1\le p_1,p_2\le \infty$, and $a, a_1,a_2\ge 0$ satisfy \eqref{eq:cond}. \\
\begin{itemize}
\item If $\max\left\{0,n\pr{\f{1}{p} - 1}\right\}<s<1$ or $s\in 2\nn$,
\begin{align}
\n{J^s (fg)-f J^s g}{L^p_{ \langle x\rangle^a
}(\rr^n)} &\lesssim  \n{J^s f}{L^{p_1}_{\langle x \rangle^{a_1}}(\rr^n)} \n{g}{L^{p_2}_{\langle x \rangle^{a_2}}(\rr^n)} + \n{\nabla f}{L^{p_1}_{\langle x \rangle^{a_1}}(\rr^n)}\n{J^{s-1} g}{L^{p_2}_{\langle x \rangle^{a_2}}(\rr^n)}. \label{eq:new1.4}
\end{align}
\item If $s>\max\left\{1,n\pr{\f{1}{p} - 1}\right\}$ or $s\in 2\nn$,
\begin{align}
\n{J^s (fg)-f J^s g - s \na f \cdot \na J^{s-2} g}{L^p_{ \langle x\rangle^a
}(\rr^n)} &\lesssim  \n{J^s f}{L^{p_1}_{\langle x \rangle^{a_1}}(\rr^n)} \n{g}{L^{p_2}_{\langle x \rangle^{a_2}}(\rr^n)} + \n{\nabla f}{L^{p_1}_{\langle x \rangle^{a_1}}(\rr^n)}\n{J^{s-1}g}{L^{p_2}_{\langle x \rangle^{a_2}}(\rr^n)}. \label{eq:new1.5}
\end{align}
\end{itemize}

Moreover, if $0<s\leq n\pr{\f{1}{p}-1}$ and $s\not\in 2\nn$, both \eqref{eq:new1.4} and \eqref{eq:new1.5} fails.
\end{thm}

For the $L^\infty$ endpoint (that is, $L^\infty\times L^\infty \to L^\infty$), Theorem~\ref{main-comm} is included in \cite[Theorem 1.20]{BL}, with a minor difference in the correction term on the left hand side (LHS) of \eqref{eq:new1.5}.  On the other hand, the $L^1$ endpoint estimates (that is, $p_1=1$ or $p_2=1$) in Theorem~\ref{main-comm} are new even in the unweighted setting.  This theorem also portrays that the previously highlighted differences for \eqref{1.4} also stands for the Kato-Ponce commutator estimate, where strong type estimates hold at the $L^1$ and $L^\infty$ endpoints and the multilinear $\A_{p_1,p_2}$ condition is not necessary.

Distinction between the two inequalities \eqref{eq:new1.4} and \eqref{eq:new1.5} arises from the boundedness property of the Riesz transforms acting on $L^{p_2}_{\lan{x}^{a_2}}$.  In fact, the authors in \cite{BL} showed that \eqref{eq:new1.4} fails at the strong $L^\infty$ endpoint when $s>1$ and $s\not\in 2\nn$ due to the unboundedness of Riesz transforms in $L^\infty$.  On the other hand,  it is easy to see that \eqref{eq:new1.5} implies \eqref{eq:new1.4} if Riesz transforms are bounded in $L^{p_2}_{\lan{x}^{a_2}}$.   For $1<p<\infty$, it is known \cite{stein93} that the Riesz transforms are bounded on $L^p_w$ if and only if $w \in A_p$, which immediately leads to the following corollary.
\begin{cor}\label{cor}
Let  $\frac{1}{2}< p < \infty$, $1\le p_1\le \infty,1<p_2< \infty$, and let $0\le a_1<\infty,0\le a_2\le n(p_2-1),$ and $a/p=a_1/p_1+a_2/p_2$. If  $s>\max\left\{1,n\pr{\f{1}{p} - 1}\right\}$ {and $s\not\in 2\nn$}, then \eqref{eq:new1.4} holds.
\end{cor}

We will also consider an extension of \eqref{1.4} to allow mixed-norms and biparameter fractional derivatives.  Given indices $\dot{p},\ddot{p} \in (0,\infty]$ and $\dot{a},\ddot{a} \in [0,\infty)$ as well as dimensional indices $\dot{n}$, $\ddot{n}$, we will define the mixed $L^{\ddot{p}}_{\ddot{a}} L^{\dot{p}}_{\dot{a}}$ (quasi-)norm:
\[
\|f\|_{L^{\ddot{p}}_{\ddot{a}} L^{\dot{p}}_{\dot{a}}} = \left(\int_{\rr^{\ddot{n}}}\left(\int_{\rr^{\dot{n}}} |f(x,x')|^{\dot{p}} \langle x \rangle^{\dot{a}} d\dot{x}\right)^{\f{\ddot{p}}{\dot{p}}} \langle \ddot{x}\rangle^{\ddot{a}} d\ddot{x}\right)^{\f{1}{\ddot{p}}}.
\]

The following two theorems are extensions of our results in \cite{OW} to the weighted mixed norm setting.  We first state the mixed norm variant:

\begin{thm}\label{main2}
Let $\dot{p}, \ddot{p} \in \left[\f{1}{2},\infty\right]$,  $\dot{a}, \ddot{a}, \dot{a_1},\ddot{a_1},\dot{a_2},\ddot{a_2}\in [0,\infty)$, and $\dot{p}_1, \ddot{p}_1,\dot{p}_2,\ddot{p}_2\in [1,\infty]$ satisfy
\begin{equation}\label{eq:cond2}
\f{1}{\dot{p}} = \f{1}{\dot{p}_1} + \f{1}{\dot{p}_2},\qquad \f{1}{\ddot{p}} = \f{1}{\ddot{p}_1} + \f{1}{\ddot{p}_2},\qquad\f{\dot{a}}{\dot{p}} = \f{\dot{a}_1}{\dot{p}_1} + \f{\dot{a}_2}{\dot{p}_2},\qquad \f{\ddot{a}}{\ddot{p}} = \f{\ddot{a}_1}{\ddot{p}_1} + \f{\ddot{a}_2}{\ddot{p}_2},
\end{equation}
We denote $p^* = \min\{1,p, p'\}$.  If $s>\max\left\{0,(\dot{n}+\ddot{n})\pr{\f{1}{p^*} - 1}\right\}$ or $s\in 2\nn$,
\begin{align}
\n{J^s (fg)}{\lp } &\lesssim  \n{J^s f}{\lpa} \n{g}{\lpb } + \n{f}{\lpa }\n{J^s g}{\lpb }. \label{eq:1.3b}
\end{align}
Moreover, the range of $s$ above is optimal. More precisely, if $s\le \max\{0,(\dot{n}+\ddot{n})\pr{\f{1}{p^*} - 1}\}$ and $s\notin 2\nn,$ then the inequality above fails in general.
\end{thm}

This mixed norm inequality with the homogeneous differential operator $D^s$ has been studied by Torres and Ward \cite{TW} for the Banach range of indices and then by Hart, Torres and the second author of this manuscript \cite{HTW} for the quasi-Banach range of indices.   The extension to the full range of Lebesgue indices in unweighted setting was recently obtained by the authors of this manuscript in \cite{OW}.

Finally, the biparameter variant of \eqref{1.4} is formulated as follows.
\begin{thm}\label{main3}
Let the parameters be as defined in Theorem~\ref{main2}.  Let  $\dot{s}>\max\left\{0,\dot{n}\pr{\f{1}{\bar{p}} - 1}\right\}$ or $\dot{s}\in 2\nn$, where $\bar{p} := \min\{1,\dot{p}\}$.  Also, let  $\ddot{s}>\max\left\{0,\ddot{n}\pr{\f{1}{p^*} - 1}\right\}$ or $\ddot{s}\in 2\nn$, where $p^* = \min \{1,\dot{p},\ddot{p}\}$.  Define $J^{\dot{s}}$ to be the operator acting on $\dot{x}\in \rr^{\dot{n}}$ and  $J^{\ddot{s}}$ to be the operator acting on $\ddot{x}\in \rr^{\ddot{n}}$.  Then,
\begin{align}
\n{J^{\dot{s}}J^{\ddot{s}} (fg)}{\lp } &\lesssim  \n{J^{\dot{s}}J^{\ddot{s}} f}{\lpa}\n{g}{\lpb} + \n{J^{\dot{s}} f}{\lpa}\n{J^{\ddot{s}}g}{\lpb} \label{eq:1.3c}\\
&\quad + \n{J^{\ddot{s}}f}{\lpa}\n{J^{\dot{s}}g}{\lpb}+ \n{f}{\lpa}\n{J^{\dot{s}}J^{\ddot{s}} g}{\lpb}.\notag
\end{align}
\end{thm}

The biparameter Kato-Ponce inequality has been extensively studied in the literature, where often the homogeneous operator $D^s$ is used instead of $J^s$.  Such inequality was first considered by Muscalu, Pipher, Tao and Thiele \cite{MPTT} and also by Grafakos and the first author of this manuscript \cite{GO} in the unweighted Lebesgues spaces without the additional complexity of mixed-norm spaces. This result was then used by Kenig \cite{K} in the analysis of the KP-I equation. Benea and Muscalu \cite{BM1} studied the biparameter inequality in context of the mixed Lebesgue spaces when the target space, $L^{p'}L^p$, lied within the Banach range. This result was extended by Di Plinio and Ou \cite{DO} and also by Benea and Muscalu \cite{BM2} to allow for quasi-Banach Lebesgue exponents for the target space.  Finally, the extension to the full range of Lebesgue indices was recently obtained in \cite{OW}.

This article is organized as follows.  In Section~\ref{sec:prelim}, we introduce standard notations.  Section~\ref{sec:inhom} contains the proof of Theorem~\ref{main1}, as well as a few technical lemmas which will be useful in subsequent sections.  This section is further divided into subsections: Subsection~\ref{sec:inhom1} contains technical lemmas; Subsection~\ref{sec:inhom2} contains the proof of the positive direction in Theorem~\ref{main1}; Section~\ref{sec:inhom3} contains the proof of the negative direction in Theorem~\ref{main1}.   Section~\ref{sec:com} contains the proof of Theorem~\ref{main-comm} as well as Corollary~\ref{cor}.  Subsections are organized as follows: Subsection~\ref{sec:com1} contains technical lemmas; Subsection~\ref{sec:com2} contains the proof of \eqref{eq:new1.4};  Subsection~\ref{sec:com3} contains both the proof of \eqref{eq:new1.5}; Subsection~\ref{sec:com4} contains the proof of the negative direction of Theorem~\ref{main-comm}.  Finally, Section~\ref{sec:multi} contains necessary and non-trivial modifications to be made from the proof in \cite{OW} to establish Theorems~\ref{main2} and \ref{main3}.

\section{Preliminary Notations} \label{sec:prelim}
In this section, we introduce  notations  which we will use in the sequel.

For two positive quantities $A$ and $B$, we write $A\lesssim B$ if there exists an absolute constant $C$, which does not depend on main parameters, such that $A\leq C B$.
 When this constant depends on some parameters, for instance $\ve$, we will sometimes indicate this dependence by writing $A \lesssim_\ve B$.   $A\sim B$ means that both $A\lesssim B$ and $B\lesssim A$ holds. Let  $A\gg B$ mean  $A\ge CB$ for some large positive constant $C$. 

Let $C_0^\infty(\rr^n)$ be the space of smooth functions with compact support, and $\S(\rr^n)$ be the Schwartz class of smooth, rapidly decreasing functions.
For any $f\in \cs(\rr^n)$, define the Fourier transform by $\wh{f}(\xi)= \cf[f](\xi) := \int_{\rr^n} f(x) e^{ - i x\cdot \xi}\, dx$. For $s\in \rr$ and $\xi \in \rr^n$, we denote $\langle \xi \rangle^s=(1+|\xi|^2)^{s/2}$. The fractional differential operators $D^s$ and $J^s$ are defined respectively by  $\wh{D^s f}(\xi) := |\xi|^s \wh{f}(\xi)$ and $\wh{J^sf}(\xi):=\langle \xi \rangle^s \wh{f}(\xi)$ for $s\in \rr$ and $f\in \S(\rr^n)$.

Next, we define the notations for the Littlewood-Paley frequency localization operators $\De_j$ and $S_j$.  We begin with a particular radial non-negative smooth function $\Phi$ on $\rr^n$, such that $\Phi (\xi) = 0$ when $|\xi|\geq 2$, and $\Phi(\xi) =1$ when $|\xi|\leq 1$.  Also, define $\Psi(\xi) = \Phi(\xi) - \Phi(2\xi)$.  With these functions, we define the  frequency localization operators:
\begin{align*}
    \wh{\De_j f}(\xi)&=\wh{\De_j^{\Psi} f}(\xi)  := \Psi( 2^{-j} \xi) \wh{f}(\xi),\\
    \wh{S_j f}(\xi) &=\wh{S_j^{\Phi} f}(\xi) := \Phi(2^{-j} \xi) \wh{f}(\xi),\\
    \wt{\De_j} f  &:= \sum_{k:|k-j|<3} \De_k f.
\end{align*}
We will usually omit the superscripts $\Phi$ and $\Psi$ when  used for localization, since these functions are implicit in the operator's definitions.  Sometimes, we will need to replace these functions with  their slightly modified versions.  In such cases, the functions in the superscript will indicate the expression used for localization.  For instance, $S^\phi$ means that the function $\phi$ replaces $\Phi$ in the definition of multiplier.  Still, the support of $\phi$ will be restricted to a ball centered at the origin of a size equivalent to 1, while it may not equal to 1 on a smaller ball.

The following reproducing formula holds (cf. \cite{Grafakos2})
\begin{equation*}
f(x)=S_0 f(x) + \sum_{k\in \nn} \De_j f(x),\qquad \forall \ f\in \S(\rr^n),
\end{equation*}
where the above series converges in $\S(\rr^n).$ By duality, the above formula holds also for $f\in \S'(\rr^n)$,
and the series converges in the topology of $\S'(\rr^n).$

\section{Inhomogeneous Kato-Ponce Inequality with Polynomial Weights}\label{sec:inhom}

\subsection{Lemmas}\label{sec:inhom1}

In this section, we will present technical tools necessary to work in polynomially-weighted Lebesgue spaces.  A few of the lemmas given in this section may already exist in the literature, but we include their proofs unless an explicit reference is available.

We begin with Young's inequality for polynomially-weighted $L^p$ spaces. It is worthwhile to point out that this inequality does not require the Muckenhoupt weight condition.

\begin{lem}[Weighted Young's inequality]\label{Young}
Let $1\le p,q,r\le \infty$ satisfy $1/r+1=1/p+1/q.$ Then for all $a\ge 0,$
$$\|f*g\|_{L^r_{\langle x \rangle^a}}\le \|f\|_{L^p_{\langle x \rangle^a}} \|g\|_{L^q_{\langle x \rangle^a}}.$$
\end{lem}
\begin{proof}
In this proof, we temporarily use $p'$ to denote the conjugate exponent of $p$ for $1\le p\le \infty$, which is given by $1/p'+1/p=1.$

If $q=\infty,$ then $r=\infty$ and $p=1$, and the conclusion hold trivially.  We omit the details in this case.

If $q<\infty,$ then by our condition on the indices can be written as
\[
\frac{1}{q'}+\frac{1}{r}+\frac{1}{p'}=1,\qquad \frac{p}{r}+\frac{p}{q'}=1,\qquad \frac{q}{r}+\frac{q}{p'}=1.
\]

H\"{o}lder's inequality with exponents $q',r,p'$ yields
\begin{align*}
  (|f|*|g|)(x) & =\int_{\rr^n} |f(y)|^{\frac{p}{q'}} (|f(y)|^{\frac{p}{r}} |g(x-y)|^{\frac{q}{r}}) |g(x-y)|^{\frac{q}{p'}} dy\\
 & \le \left(\int_{\rr^n} |f(y)|^p |g(x-y)|^q dy\right)^{\frac{1}{r}}\|f\|_{L^{p}}^{\frac{p}{q'}}  \|g\|_{L^{q}}^{\frac{q}{p'}}.
\end{align*}
Now taking the $L^r_{\langle x \rangle^a}$ norm, applying Fubini's theorem and the inequality
$$\langle x\rangle\le \langle y\rangle \langle x-y\rangle,\qquad \forall\ x,y\in \rr^n,$$
 we deduce that
\begin{align*}
  \left\||f|*|g|\right\|_{L^r_{\langle x \rangle^a}} & \le \left(\int_{\rr^n\times \rr^n} |f(y)|^p |g(x-y)|^q dy \langle x\rangle^a dx\right)^{\frac{1}{r}}\|f\|_{L^{p}}^{\frac{p}{q'}}  \|g\|_{L^{q}}^{\frac{q}{p'}} \\
  &\le \left(\int_{\rr^n\times \rr^n} |f(y)|^p\langle y\rangle^a |g(x-y)|^q  \langle x-y\rangle^a dx dy\right)^{\frac{1}{r}}\|f\|_{L^{p}}^{\frac{p}{q'}}  \|g\|_{L^{q}}^{\frac{q}{p'}} \\
  &= \|f\|_{L^{p}_{\langle x \rangle^a}}^{\frac{p}{r}}  \|g\|_{L^{q}_{\langle x \rangle^a}}^{\frac{q}{r}} \|f\|_{L^{p}}^{\frac{p}{q'}}  \|g\|_{L^{q}}^{\frac{q}{p'}}\\
   &= \|f\|_{L^{p}_{\langle x \rangle^a}} \|g\|_{L^{r}_{\langle x \rangle^a}},
\end{align*}
concluding the proof.
\end{proof}

Lemma~\ref{Young} leads to the $L^p_{\lan{x}^{a}}$ boundedness of Littlewood-Paley operators for $p\in [1,\infty]$, which is analogous to the corresponding estimates for unweighted spaces.

\begin{lem}[Weighted bounds for the Littlewood-Paley operators]\label{lp}
Let $a\ge 0$ and $p\in [1,\infty]$. We have the following uniform bound for the Littlewood-Paley operators:
$$\sup_{k\in \nn} \pr{\|\De_k f\|_{L^p_{\langle x \rangle^a}}+\|S_k f\|_{L^p_{\langle x \rangle^a}}} \lesssim \|f\|_{L^p_{\langle x \rangle^a}}.$$
\end{lem}

\begin{proof}
We can write
\[
\Delta_k  f = 2^{kn}\wh{\Psi} (2^k \cdot) * f, \qquad
S_k  f = 2^{kn}\wh{\Phi} (2^k \cdot) * f.
\]
Applying Young's inequality from Lemma \ref{Young}, we get
\[
\n{\Delta_k f}{L^p_{\lan{x}^a}} \lesssim \n{2^{kn} \wh{\Psi}(2^k\cdot)}{L^1_{\lan{x}^{a}}} \n{f}{L^p_{\lan{x}^a}},
\]
and $S_k f$ can be bounded similarly.    Hence, it suffices to show that $\ds \n{2^{kn} \wh{\Psi}(2^k\cdot)}{L^1_{\lan{x}^{a}}} $ is uniformly bounded for $k\in\nn$.  Note that
\[
\n{2^{kn} \wh{\Psi}(2^k\cdot)}{L^1_{\lan{x}^{a}}}  = \n{ \wh{\Psi}}{L^1_{\lan{2^{-k}x}^{a}}} \leq \n{ \wh{\Psi}}{L^1_{\lan{x}^{a}}},
\]
which is uniformly bounded in $k\in \nn$.  This proves the desired estimate.
\end{proof}

Another implication of Lemma~\ref{Young} is the Bernstein inequality in the weighted Lebesgue spaces.

\begin{lem}[Weighted Bernstein inequality] \label{Bern}  Let $s> 0,$ $1\le p \le \infty,$ $k\in \nn,$
and $a\ge 0.$  Then
\begin{enumerate}
   \item[(i)] $\|\De_{>k}f\|_{L^p_{\langle x \rangle^a}(\rr^{n})} \lesssim_{a,s} 2^{-ks} \|J^s f\|_{L^p_{\langle x \rangle^a}(\rr^{n})}$
  \item[(ii)] $\|\De_{k}J^{\pm s}f\|_{L^p_{\langle x \rangle^a}(\rr^{n})} \sim_{a,p,s} 2^{\pm ks} \|\De_k f\|_{L^p_{\langle x \rangle^a}(\rr^{n})}$
\item[(iii)] $\|S_{k}J^{s}f\|_{L^p_{\langle x \rangle^a}(\rr^{n})} {\lesssim_{a,p,s}} 2^{ks} \| f\|_{L^p_{\langle x \rangle^a}(\rr^{n})}$
\item[(iv)] for any $0\leq \ve <\min (1,s)$, we have
\[
\|\nabla S_k f\|_{L^p_{\langle x \rangle^a}} \lesssim_\ve  2^{k(1-\varepsilon)} \|f\|_{L^p_{\langle x \rangle^a}}^{1-\frac{\varepsilon}{s}}\|J^s f\|_{L^p_{\langle x \rangle^a}}^{\frac{\varepsilon}{s}}
\]
\end{enumerate}
\end{lem}

\begin{proof}
For (i), applying Lemma \ref{Young}, we have
\begin{align*}
  \|\De_{>k}f\|_{L^p_{\langle x \rangle^a}} & =\n{\sum_{j>k}\De_j f}{L^p_{\langle x \rangle^a}} =\n{\sum_{j>k}\De_j (J^{-s}J^sf)}{L^p_{\langle x \rangle^a}}\\
   &\le \sum_{j>k}2^{-js}  \n{\De_j^{\Psi^{-s}_j} J^sf}{L^p_{\langle x \rangle^a}} \le \sum_{j>k}2^{-js} \n{2^{jn}\wh{\Psi_j^{-s}}(2^j\cdot)}{L^1_{\langle x \rangle^a}} \| J^sf\|_{L^p_{\langle x \rangle^a}} , 
\end{align*}
where  $\Psi_j^{-s}(\xi) :=\Psi(\xi) (2^{-2j}+|\xi|^2)^{-\frac{s}{2}}.$  To prove the desired estimate, it suffices to show that $\n{2^{jn}\wh{\Psi_j^{-s}}(2^j\cdot)}{L^1_{\langle x \rangle^a}}$ is uniformly bounded in $j\in \nn$.  Note that, for each $\al \in \nn^n$,
$\partial^\al \Psi_j^{-s}$ is uniformly bounded in $j\in \nn$ due to the support of $\Psi$.  More precisely, direct computations show that
\[
\abs{\partial^\al  \Psi_j^{-s}\pr{2^{-j}\cdot}(\xi)} \lesssim 2^{-j|\al|}, \quad \forall \ j\in \nn, \al\in \nn^n,\ \xi\in \rr^n.
\]
In particular, for any $N\gg 1$, we have
\begin{align*}
 \abs{(1-2^{2j}\Delta)^N \Psi^{-s}_j(2^{-j}\cdot)}(\xi) &\leq  \sum_{\ell=0}^N  \pr{\begin{array}{c} N\\ \ell\end{array}}  2^{2\ell j}\abs{\Delta^\ell \Psi^{-s}_j(2^{-j}\cdot)(\xi)} \lesssim  \sum_{\ell=0}^N  \pr{\begin{array}{c} N\\ \ell\end{array}}  2^{2\ell j} 2^{-2\ell j} = 2^N.
\end{align*}
By the Hausdorff-Young inequality, we can show that
\[
\left|\wh{\Psi_j^{-s}}\pr{2^j x}\right| \lesssim_N \f{1}{(1+|2^j x|^2)^N}, \qquad \tn{ for any }j, N\in \nn, x \in \rr^n.
\]
Then,
\begin{equation}\label{eq:psisj}
\n{2^{jn} \wh{\Psi_j^{-s}}\pr{2^j \cdot} }{L^1_{\langle x \rangle^a} } \lesssim_N 2^{jn}\int_{\rr^n} \frac{(1+|x|^2)^{\frac{a}{2}}}{(1+|2^jx|^2)^{N}}\,dx = \int_{\rr^n} \frac{(1+|2^{-j} x|^2)^{\frac{a}{2}}}{(1+|x|^2)^{N}}\,dx \leq \int_{\rr^n} \lan{x}^{-\f{N-a}{2}}\, dx.
\end{equation}
Choosing $N$ so that $\f{N-a}{2} > n,$ the right hand side (RHS) above is uniformly bounded in $j\in \nn$.   This proves (i).

The proof of  (ii) follows similarly when we note that the \eqref{eq:psisj} holds for any $s\in \rr$.  We omit the details.

For (iii), we will first consider this estimate when $k=0$.  In this case, the operator $S_0 J^s$ is given by a convolution with  $J^s \wh{\Phi} \in \cs(\rr^n)$, which is bounded on $L^1_{\lan{x}^a}$
by Lemma \ref{Young}.  Then, Lemma~\ref{Young} produces the desired estimate when $k=0$.  When $k>0$, we can apply estimate (ii) to get
\begin{align*}
\n{S_{k}J^{s}f}{L^p_{\langle x \rangle^a}} &\leq \n{S_0 J^sf}{L^p_{\langle x \rangle^a}} + \sum_{j=0}^k  \|\De_j J^s f\|_{L^p_{\langle x \rangle^a}}\\
 &\sim_{a,p,s} \n{S_0 J^sf}{L^p_{\langle x \rangle^a}} + \sum_{j=0}^k 2^{js} \|\Delta_j  f\|_{L^p_{\langle x \rangle^a}}
 \end{align*}
The first term on the RHS can be treated as above for $k=0$, while second term can be treated using Lemma~\ref{lp}.  The desired estimate follows easily.

Finally, let us consider (iv).  This estimate was used in the proof of endpoint Kato-Ponce inequality in the unweighted setting in \cite{BL} where the authors used the decomposition $S_{k-3}f=\sum_{j\le k-3} \De_j f$.  However, the low-frequency pieces do not behave well in the weighted setting, so we write instead:
\[S_{k}f=\sum_{1\le j\le k} \De_j f +S_0 f.\]
Now,
\begin{align*}
\begin{split}
\| \nabla S_{k}f \|_{L^{p}_{\langle x \rangle^a}}&\le \sum_{j=1}^{k}\| \nabla \De_{j}f \|_{L^{p}_{\langle x \rangle^a}} + \| \nabla S_0 f \|_{L^{p}_{\langle x \rangle^a}}=: I_1+I_2.
\end{split}
\end{align*}

For $I_1$, we select $\wt{\Psi} \in \cs(\rr^n)$ satisfying $\wt{\Psi}\cdot \Psi = \Psi$ and write
\begin{equation}\label{eq:conv}
\n{\na \De_j f}{L^p_{\langle x\rangle^a}} = \n{2^{jn}\na \pr{\wt{\Psi}(2^j\cdot) * \De_j f}}{L^p_{\langle x\rangle^a}} \lesssim 2^j \n{2^{jn}(\na \wt{\Psi})(2^j\cdot)}{L^1_{\lan{x}^a}} \n{ \De_j f}{L^p_{\langle x\rangle^a}} \lesssim 2^j  \n{\De_j f}{L^p_{\langle x\rangle^a}},
\end{equation}
where the boundedness of $\n{2^{jn}(\na \wt{\Psi})(2^j\cdot)}{L^1_{\lan{x}^a}}$ follows since $\na \wt{\Psi} \in [\cs(\rr^n)]^n$ and $j\in \nn$. Applying Lemma~\ref{lp} and the statement (ii) from the present lemma, we can write, for any $0\leq \ve<\min(1,s)$,
\[
I_1 \lesssim \sum_{j=1}^{k} 2^{j(1-\ve)}  2^{j\ve}\|\Delta_j f\|_{L^p_{\langle x \rangle^a} }\lesssim 2^{k(1-\ve)}\sup_{j\in\nn} \|\Delta_j f\|_{L^p_{\langle x \rangle^a}}^{\frac{s-\ve}{s}} \|\Delta_j J^s f\|_{L^p_{\langle x \rangle^a}}^{\frac{\ve}{s}}\lesssim 2^{k(1-\ve)} \|f\|_{L^p_{\langle x \rangle^a}}^{\frac{s-\ve}{s}} \|J^s f\|_{L^{p}_{\langle x \rangle^a}}^{\frac{\ve}{s}}.
\]

For $I_2$, we have
\begin{align*}
\begin{split}
I_2 &= \n{\na \wh{\Phi}* f }{L^{p}_{\langle x \rangle^{a}}}^{1-\frac{\varepsilon}{s}} \n{\na J^{-s} \wh{\Phi}* J^s f }{L^{p}_{\langle x \rangle^{a}}}^{\frac{\varepsilon}{s}}.
\end{split}
\end{align*}
Since $\na  \wh{\Phi}, \,\na J^{-s} \wh{\Phi}\in [\mathcal{S}(\rr^n)]^n$, the desired estimate follows from Lemma~\ref{Young}.
\end{proof}

The weighted Young and Littlewood-Paley inequalities are helpful for proving the Kato-Ponce estimates for the Banach range of indices, that is $p\in [1,\infty]$.  However, they are not sufficient for the quasi-Banach range $p\in [\f{1}{2},1)$.  For these cases, we first recall useful decay estimates established by Muscalu and Schlag \cite[Chapter 2]{MS13a} and also by Grafakos and the first author of this manuscript \cite[Lemmas 1 and 2]{GO}.
\begin{lem}\label{le:GO} \cite{MS13a, GO}
Let $J_\delta^s = (\de^2 - \Delta)^{s/2}$ for $\de\in(0,1]$ and let $f\in \cs(\rr^n)$ and $s>0$.  Then,
\begin{equation}\label{eq:2.1}
\left|J_\de^{s} f(x)\right| \lesssim_{n,s,f}  (1+ |x|)^{-n-s}.
\end{equation}

If $s\notin 2\nn$ and $f(x)\ge 0,\forall x\in \rr^n$, and $f\not\equiv 0$, then there exists $R\gg 1$ such that
\begin{equation}\label{eq:sharp}
|D^s f(x)|\gtrsim_{n,s,f} |x|^{-n-s},\qquad \forall x: |x|>R.
\end{equation}
\end{lem}

While we will at times use this lemma in its current form, this decay rate is not sufficient to establish the desired estimate for the sharp range of $s$.  In fact, this decay estimate can lead to the proof of Kato-Ponce inequalities only when $s> \max\left\{\f{a}{p} + n \pr{\f{1}{p} - 1},0\right\}$, rather than the sharp range $s> \max\left\{n \pr{\f{1}{p} - 1},0\right\}$.  To clarify this point, we demonstrate the failure of the Kato-Ponce inequality in the following:

\begin{rem}\label{rem:counter}
(i)
When $0<s\leq \f{a}{p} + n \pr{\f{1}{p} - 1}$ and $s\not\in 2\nn$, the homogeneous Kato-Ponce inequality, where $J^s$ is replaced by $D^s$ in \eqref{eq:1.3}, fails in general 

(ii) When $0<s\leq \f{a}{p} + n \pr{\f{1}{p} - 1}$ and $s\not\in 2\nn$, the inhomogeneous Kato-Ponce inequality \eqref{1.4} fails for $w_1=|x|^{a_1}\in A_{p_1}, w_2=|x|^{a_2}\in A_{p_2}.$
\end{rem}

\begin{proof}

We will verify that the counterexample from \cite{GO} is sufficient to show the statement (i).

Assume that  $0<s\le \frac{a}{p}+n(\frac{1}{p}-1)$ and $s\notin 2\nn.$  We define $f(x):=\wh{\Phi}(x)e^{  10i e_1 \cdot x}$ and $g(x):= \wh{\Phi}(x)e^{- 10i e_1\cdot x}$ where $e_1$ is a unit vector in $\rr^n$ and $\Phi$ is as defined in Section~\ref{sec:prelim}.   Note that the supports of $\wh{f}$ and $\wh{g}$ do not contain the origin.  This means $D^s f, D^s g \in \cs (\rr^n)$ and the RHS of the homogeneous Kato-Ponce inequality is finite using these functions.

However, the support of $\wh{fg}$ does contain the origin, and $fg = \pr{\wh{\Phi}}^2$ is a nonnegative and nonzero Schwartz function. By \eqref{eq:sharp}, we know that
\[ |D^s(fg)(x)| \gtrsim |x|^{-n-s} \qquad \tn{ when } |x| \gg 1,\]
which implies
\begin{equation}\label{eq:div}
\|D^s(fg)\|_{L^p_{| x |^a}}^p \gtrsim \int_{|x|\gg 1}\frac{dx}{|x|^{p(n+s)-a}}.
\end{equation}
Since $0<s\le \frac{a}{p}+n(\frac{1}{p}-1)$, we know
\[p(n+s)-a=p\left(n+s-\frac{a}{p}\right)\le p\left(n+n\left(\frac{1}{p}-1\right)\right)\le n.\]
Hence, the RHS of \eqref{eq:div} diverges, which implies that the LHS of the homogeneous Kato-Ponce inequality is infinite with these functions.  This leads to a contradiction.

For (ii), let $f,g \in \cs (\rr^n)$ such that the support of $fg$ is nonnegative and nonzero.  We denote $f_\de (x) = f\pr{\f{ x}{\de}}$ and similarly $g_\de(x) = g\pr{\f{x}{\de}}$.
Assume that \eqref{1.4} holds with $w_1=|x|^{a_1}\in A_{p_1}, w_2=|x|^{a_2}\in A_{p_2}$.  This means that $a_i \in (-n, n(p_i - 1))$ for $i=1,2$.  

Applying \eqref{1.4} to these functions will lead to
\begin{equation*} 
\n{J^s (f_\de g_\de)}{L^p_{ |x|^a
}} \lesssim  \n{J^s f_\de}{L^{p_1}_{|x|^{a_1}}} \n{g_\de}{L^{p_2}_{|x|^{a_2}}} + \n{f_\de}{L^{p_1}_{|x|^{a_1}}}\n{J^s g_\de}{L^{p_2}_{|x|^{a_2}}},
\end{equation*}
which is equivalent to
\begin{equation} \label{homsharp}
\n{ J_\de^s (f g) }{L^p_{ |x|^a}} \lesssim \n{J^s_\de f}{L^{p_1}_{|x|^{a_1}}} \n{g}{L^{p_2}_{|x|^{a_2}}} + \n{ f}{L^{p_1}_{|x|^{a_1}}} \n{J^s_\de g}{L^{p_2}_{|x|^{a_2}}}.
\end{equation}

For the RHS, we note that
\begin{align*}
\int_{\rr^n} |J_\de^s (f )(x)|^{p_1}  |x|^{a_1} \, dx &\lesssim_{f, M}  1+ \int_{|x|\geq 1} \pr{\lan{x}^{-M} + |x|^{-(n+s)p_1+ a_1 }} e^{-c \de |x|} \, dx\\
 &\lesssim_M 1+\de^{-n + (n+s)p_1 -a_1} \int_{|x|\geq \de} |x|^{-(n+s)p_1+ a_1 }e^{-c |x|} \, dx
\end{align*}
Since we have assumed $a_1 < n(p_1 - 1)$ and $s>0$, we have $-n + (n+s) p_1 - a_1 >0$, which is equivalent to $-(n+s)p_1 + a_1 < -n.$  This means that the exponent of $\de$ is positive, while the integrand is not locally integrable.  Using polar coordinates, we can bound this expression by a constant multiple of $1+ \de^{-n + (n+s)p_1 -a_1}$, which is uniformly bounded in $\de \in (0,1]$.  Similarly, the second term on the RHS of \eqref{homsharp} is uniformly bounded in $\de$.

On the other hand,  the dominated convergence theorem implies that $J_\de^s (fg) \to D^s (fg)$ almost everywhere.  Hence, by Fatou's lemma,  
\[
\left(\int_{\rr^n} |D^s (fg)|^p (x) \, |x|^a\, dx \right)^{\f{1}{p}} \leq \liminf_{\de \to 0}\left(\int_{\rr^n} |J_\de^s (fg)|^p (x) \, |x|^a\, dx \right)^{\f{1}{p}}.
\]
But $|D^s (fg)|(x) \gtrsim \lan{x}^{-n-s}$ for $|x|\gg 1$, so that the integrand on the LHS is bounded below by a constant multiple of $\lan{x}^{-(n+s)p + a}$ when $|x| \gg 1$. Noting that
\[
s \leq n\pr{\f{1}{p}-1} + \f{a}{p} \iff -( n+s)p + a \geq -n,
\]
we conclude that the LHS of \eqref{homsharp} is not uniformly bounded in $\de\in (0,1]$, while the RHS remains uniformly bounded.  This  leads to a contradiction. 
\end{proof}

While we do not prove the positive directions of these Kato-Ponce inequalities in this manuscript, these counterexamples bring up a few noteworthy points:

\begin{itemize}
    \item The decay rate given in \eqref{eq:2.1} is insufficient for the proof of Theorem~\ref{main1} for the full sharp range of indices $s > \max\left\{n\pr{\f{1}{p}-1},0\right\}$, while it may be sufficient for a smaller range $s > \max\left\{\f{a}{p} + n\pr{\f{1}{p}-1},0\right\}$.

    \item A theorem in \cite{NT} states that homogeneous ($D^s$) and inhomogeneous ($J^s$) Kato-Ponce inequalities with weights satisfying the $A_p$ condition is valid for $s>\max\left\{n\pr{\f{\tau(w)}{p}-1},0\right\}$, where $\tau(w) = \inf \{\tau \in (1,\infty): w \in A_\tau\}$.  For the polynomial weights $w= \lan{x}^a$ and $w= |x|^a$, this restriction is translated to $s > \max\left\{\f{a}{p} + n\pr{\f{1}{p}-1},0\right\}$.  Hence, the counterexamples above demonstrate that both the homogeneous and inhomogeneous Kato-Ponce inequalities fails below this threshold of $s$ for certain (for instance, polynomial) $A_p$ weights.

    \item These counterexamples contradict  statements in \cite{CN} which claim the validity of both Kato-Ponce inequalities with $A_p$ weights for the larger range $s > \max\left\{n\pr{\f{1}{p}-1},0\right\}$.
\end{itemize}

Nevertheless, the full range of the inhomogeneous Kato-Ponce inequality still holds true with weights $\lan{x}^a$ as stated in Theorem~\ref{main1}.  In order to prove the inequality for the sharp range of $s$, we will need an additional exponential decay in place of \eqref{eq:2.1}, which is achieved by modifying the proof of Lemma 2 in \cite{GO}.

\begin{lem} \label{le:jsde}
Let $J^s_\de$ be as defined in Lemma~\ref{le:GO}, $s \in (-n,\infty)\setminus 2\nn$ and $f \in \cs(\rr^n)$.  For any $M\gg 1$,
\[
|J_\de^s f (x)| \lesssim_{n,s,f, M} \lan{x}^{-M} +  \lan{x}^{-n-s} e^{-c \de |x|}
\]
where the implicit constant is independent of $\de\in (0,1]$ and $x\in \rr^n$.  Further, if $s\in 2\nn$,
\[
|J_\de^s f (x) | \lesssim_{n,s,f, M} \lan{x}^{-M}.
\]
\end{lem}

\begin{proof}
We begin with the expression \cite[Lemma 2]{GO}, which was obtained using an analytic continuation argument:  If $s <0$, $J^s_\de$ is a convolution operator given by $J^s_\de f = K^{\de}_s *f$
where the convolution kernel is defined by
\begin{equation}\label{eq:ksde}
K_s^{\de}(y) :=  \int_0^{\infty} e^{-\f{|y|^2}{t}} e^{-\de^2 t} t^{-\f{s+n}{2}}\, \f{dt}{t}.
\end{equation}
This convolution kernel is locally integrable as long as $s<0$, which means that the convolution is well-defined.  If $s\geq 0$, an analytic continuation arguments can be used to extend this definition.  For instance, given $s\in [N-1,N)$ where $N\in \nn$, we can define
\begin{equation}\label{bessel}
(J^s_\de f)(x) =
C(s) \int_{\rn}K^{\de}_s(y)\, \left[f(x+y) - \sum_{|\al| < N} \f{[\p^{\al} f](x)}{\al !} y^{\al}\right]\, dy +\sum_{|\al| < N} C(\al, n) \de^{s-|\al|}[\p^{\al} f](x)\, ,
\end{equation}
as the holomorphic continuation.  Here, $C(s)$ is a holomorphic function of $s$ with zeros on $2\nn$.

We can see that the second term on the RHS of \eqref{bessel} only contains the local derivatives of $f$.  Since we assumed that $f$ is a Schwartz function, this term decays faster than any polynomial, which only contributes to the $\lan{x}^{-M}$ term for our estimate.  This proves the desired  estimate in the case $s\in 2\nn$ since now the second term on the RHS of \eqref{bessel} is the only nonzero term.

Let us consider the first term on the RHS of \eqref{bessel}.  {It was shown in \cite{GO} that the only part of this expression which does not decay like a Schwartz function is the following term:}
\begin{equation*}
C(s) \int_{|y|\geq 1}K^{\de}_s(y)\, f(x+y)\, dy.
\end{equation*}
While this argument was constructed primarily for $s\geq 0$, it is equally valid for $s< 0$ since the expression for  $J^s_\de f$ is free of  non-local-integrability issues to begin with (hence is simpler).  The restriction $|y|\geq 1$ guarantees that we are away from the singularity of $K_s^\de(y)$, so that $|K^s_\de(y)|\lesssim 1$ within this domain uniformly in $\de \in (0,1]$.    We can split this integral as follows:
\[
\int_{|y|\geq 1, |x+y|\ll |x|}K^{\de}_s(y)\, f(x+y)\, dy + \int_{|y|\geq 1, |x+y|\gtrsim |x|}K^{\de}_s(y)\, f(x+y)\, dy  =: I + II.
\]

For $II$, we can see that any decay in $|x+y|$ from the integrand leads to a similar decay in $|x|$.  Hence,
\[
|II| \lesssim \lan{x}^{-M} \int_{\rr^n} \lan{x+y}^M |f(x+y)|\, dx,
\]
which again means that $II$ decays like a Schwartz function.

It remains to prove the desired  bound for $I$.   We will achieve this bound by first estimating $K_s^\de(y)$.  Splitting the integral in \eqref{eq:ksde} at some value $B$ to be determined later, we can write
\begin{equation}\label{eq:ksde1}
K_s^{\de}(y) = \int_0^{B} e^{-\f{|y|^2}{t}} e^{-\de^2 t} t^{-\f{s+n}{2}}\, \f{dt}{t} + \int_B^{\infty} e^{-\f{|y|^2}{t}} e^{-\de^2 t} t^{-\f{s+n}{2}}\, \f{dt}{t}.
\end{equation}
For the first integral in \eqref{eq:ksde1}, we write
\[
\int_0^{B} e^{-\f{|y|^2}{t}} e^{-\de^2 t} t^{-\f{s+n}{2}}\, \f{dt}{t} \leq \int_0^{B} e^{-\f{|y|^2}{t}} t^{-\f{s+n}{2}}\, \f{dt}{t}= |y|^{-(n+s)} \int_0^{\f{B}{|y|^2}} e^{-\f{1}{t}} t^{-\f{s+n}{2}}\, \f{dt}{t}.
\]
As long as $t^{-\f{s+n}{2} - 1}$ carries a non-positive exponent, the integrand is bounded over the domain of integration.  Hence, we can estimate this integral by the size of the domain, which leads to
\[
\int_0^{B} e^{-\f{|y|^2}{t}} e^{-\de^2 t} t^{-\f{s+n}{2}}\, \f{dt}{t} \lesssim_{s,n} B|y|^{-(n+s)-2} \qquad \tn{ if } s \geq -n-2.
\]
To estimate the second integral in \eqref{eq:ksde1}, we can use
\[
-\f{|y|^2}{t} - \de^2 t \leq -2\de |y|.
\]
This leads to
\[
\int_B^{\infty} e^{-\f{|y|^2}{t}} e^{-\de^2 t} t^{-\f{s+n}{2}}\, \f{dt}{t}  \leq e^{-2\de |y|} \int_B^\infty t^{-\f{n+s}{2} - 1}\, dt =  \f{2}{s+n} e^{-2\de |y|} B^{-\f{n+s}{2}} \qquad \tn{ if } s > -n.
\]
To find the optimal cutoff $B$, we set the two upper bounds equal to each other, which leads to
\[
B = |y|^{2} e^{-\f{4\de}{n+s+2} |y|}.
\]
Substituting this into the upper bounds gives us
\begin{equation*}
K_s^{\de}(y) \lesssim_{n,s}  |y|^{-(n+s)} e^{-\f{4\de}{n+s+2} |y|}.
\end{equation*}
Applying this estimate, we obtain
\[
|I| \lesssim \int_{|y|\sim |x|} |y|^{-n-s} e^{-c \de |y|} |f(x+y)|\, dy \lesssim |x|^{-n-s} e^{- c' \de |x|} \int_{\rr^n} |f(y)|\, dy,
\]
for some positive constants  $c$ and $c'$.  This proves the desired estimate for $I$, and hence completes the proof of Lemma \ref{le:jsde}.
\end{proof}

The improved estimate for $J^s_\de (f)$ given in Lemma~\ref{le:jsde} together with square-function estimates like the one used in \cite{CN} would lead to the validity of the inhomogeneous Kato-Ponce estimate for the sharp range of $s$, but only for non-endpoint ($L^1$ and $L^\infty$) cases.  In order to treat the endpoint cases, we will adopt the methods used in \cite{BL, OW} in the unweighted setting. The following is the main lemma for the proof of Theorem~\ref{main1}.

\begin{lem}\label{le:main}
Let  $p\in (0,\infty]$, $a\geq 0$, and let $\{\si_k\}_{k\in \nn}$ be a compactly supported family of $C^{\infty}$ functions on $\rr^n$.  That is, there exists $R\gg 1$ such that
\begin{equation}\label{new3.11}
  \tn{supp}\,\si_k \subset \{\xi\in \rr^n: |\xi|< R\}
\end{equation}
for all $k\in \nn$.  If there exist $\ga > \max\{\f{n}{p},n\}$ satisfying
\[
|\wh{\si_k}(x) | \lesssim_M \lan{x}^{-M} + \lan{x}^{-\ga} e^{-c 2^{-k} |x|}
\]
for any $M\gg 1$ and $k\in \nn$, where the implicit constant is independent of $k\in \nn$ and $x\in \rr^n$, then
\begin{equation}\label{new2.1}
\n{ \int_{\rr^n} \si_k(2^{-k} \xi) \wh{S_k h}(\xi)  e^{i\xi \cdot (\cdot)}\, d\xi }{L^p_{\lan{x}^a}} \lesssim \n{S_k h}{L^p_{\lan{x}^a}}.
\end{equation}
\end{lem}

\begin{proof}
Expanding $\si_k$ in a Fourier series on the cube $[-R,R]^n$ for each $k\in \nn$, we can write
\begin{equation}\label{eq:fseries}
\si_k (\xi) = \chi_{[-R,R]^n}(\xi) \sum_{\mm\in \zz^n} c_{\mm}^k e^{-\f{2\pi i \mm \cdot \xi}{R}},
\end{equation}
  where
\[
  c_{\mm}^k = C_{R,n} \int_{[-R,R]^n} \si_k(\xi) e^{\f{2\pi i \mm \cdot \xi}{R}}\, d\xi = C_{R,n} \wh{\si_k} \pr{\f{2\pi\mm}{R}}.
\]
The decay assumption for $\wh{\si_k}$ gives
\begin{equation}\label{eq:cmk}
\left|c_{\mm}^k \right| \lesssim_{R,M,n} \lan{\mm}^{-M} +  \lan{\mm}^{-\ga} e^{- c' 2^{-k}|\mm|}
\end{equation}
for some $c'>0$.  Using identity \eqref{eq:fseries}, we can write the integral within the norm in the LHS of \eqref{new2.1} as
\begin{align*}
\int_{\rr^n} \si_k(2^{-k} \xi) \wh{S_k h}(\xi)  e^{i\xi \cdot x}\, d\xi &= 2^{kn} \int_{\rr^n} \si_k (\xi) \Phi (\xi) \wh{h}(2^{k}\xi) e^{i2^{k} \xi\cdot x}\, d\xi\\
 &= 2^{kn}\int_{\rr^n}\sum_{\mm\in \zz^n} c^k_{\mm}  e^{-\f{2\pi i \mm \cdot \xi}{R}} \chi_{[-R,R]^n}(\xi)\Phi (\xi) \wh{h}(2^{k}\xi) e^{i2^{k} \xi\cdot x}\, d\xi \\
 &= \sum_{\mm\in \zz^n} c^k_{\mm} \int_{\rr^n} \Phi (2^{-k}\xi) \wh{h}(\xi) e^{i \xi\cdot \pr{x- \f{2\pi}{R} 2^{-k}\mm }}\, d\xi\\
&= \sum_{\mm\in \zz^n} c^k_{\mm} S_k h\pr{x- \f{2\pi}{R} 2^{-k}\mm }.
\end{align*}
It follows that
\begin{align*}
\n{ \int_{\rr^n} \si_k(2^{-k} \xi) \wh{S_k h}(\xi)  e^{i\xi \cdot (\cdot)}\, d\xi }{L^p_{\lan{x}^a}}^{\bar{p}} &\leq \sum_{\mm\in \zz^n} |c^k_\mm|^{\bar{p}} \n{ S_k h\pr{\cdot - \f{2\pi}{R} 2^{-k}\mm } \lan{\cdot}^{\f{{a}}{p}} }{L^p}^{\bar{p}}\\
 &= \sum_{\mm\in \zz^n} |c^k_\mm|^{\bar{p}} \n{ S_k h   { \lan{\cdot+ \f{2\pi}{R} 2^{-k}\mm  }^{\f{{a}}{p}}} }{L^p}^{\bar{p}}\\
 &\lesssim_R \sum_{\mm\in \zz^n} |c^k_\mm|^{\bar{p}} \n{ S_k h   \lan{\cdot}^{\f{{a}}{p}} \lan{2^{-k}\mm }^{\f{{a}}{p}} }{L^p}^{\bar{p}}\\
 &= \sum_{\mm\in \zz^n} |c^k_\mm|^{\bar{p}} \lan{2^{-k}\mm}^{\f{a\bar{p}}{p}} \n{ S_k h }{L^p_{ \lan{x  }^a}}^{\bar{p}},
\end{align*}
where $\bar{p} = \min(p,1)$.
Applying \eqref{eq:cmk}, the last term above is bounded by
\[
 \sum_{\mm\in \zz^n}\pr{\lan{\mm}^{-M\bar{p}}\lan{2^{-k}\mm}^{\f{a\bar{p}}{p}}  +  \lan{\mm}^{-\ga \bar{p}} e^{- c'\bar{p} 2^{-k}|\mm|}\lan{2^{-k}\mm}^{\f{a\bar{p}}{p}} } \n{ S_k h }{L^p_{ \lan{x  }^a}},\qquad \forall M\gg 1.
\]
The summation involving the first term in the parentheses is convergent for $M>\frac{a}{p}+\frac{n}{\bar p}$ since
\[
\lan{\mm}^{-\bar{p}M} \lan{2^{-k}\mm}^{\f{a\bar{p}}{p}} \leq \lan{\mm}^{-M\bar{p} + \f{a\bar{p}}{p}}.
\]
The second term contains the exponentially decaying factor from Lemma~\ref{le:jsde}.  This exponential factor is particularly useful in treating the translation factor $\lan{2^{-k}\mm}^{\f{a\bar{p}}{p}}$ when we observe that $e^{- c' \bar{p}2^{-k}|\mm|} \lan{2^{-k}\mm}^{\f{a \bar{p}}{p}}$ is uniformly bounded for $k\in \nn$ and $\mm \in \zz^n$.  The boundedness of this expression is what allows the Kato-Ponce estimates to hold even when the underlying spaces are not translation invariant. Without this extra exponential decay (that is, if we were merely applying Lemma~\ref{le:GO}), the decay rate $\ga$ would need to be greater that what is stated in this lemma, which would not be sufficient for the sharp range of $s$.  Hence, we obtain the bound
\[
 \n{ \int_{\rr^n} \si_k(2^{-k} \xi) \wh{S_k h}(\xi)  e^{i\xi \cdot (\cdot)}\, d\xi }{L^p_{\lan{x}^a}}^{\bar{p}}  \lesssim \sum_{\mm\in \zz^n}\pr{\lan{\mm}^{-M\bar{p}+ \f{a\bar{p}}{p}} +  \lan{\mm}^{-\bar{p}\ga}} \n{ S_k h }{L^p_{ \lan{x  }^a}}^{\bar{p}}.
\]
Since $\bar{p}\ga >n$ by our assumption, the series on the RHS converges, which leads to the desired bound.
\end{proof}

The following commutator estimate is also crucial for our proof of the main result.
Lemma~\ref{le:main} will be used to establish the majority of estimates in our proof.  One exception is the commutator estimate, which will require the following version instead.

\begin{lem}\label{le:commutator}
Let $ p\in [\f{1}{2},\infty]$, $1\le p_1,p_2\le \infty$ satisfy $\f{1}{p} = \f{1}{p_1} + \f{1}{p_2}$, and $\f{a}{p} = \f{a_1}{p_1} + \f{a_2}{p_2}$.  If $\{\si_k\}_{k\in \nn}$ is as in Lemma~\ref{le:main},
\begin{equation}\label{3.15}
\n{ \int_0^1 \int_{\rr^{n+n}} \si_k (2^{-k} (t\xi+\eta)) \wh{S_{k-3} f}(\xi) \wh{\De_k g}(\eta) e^{i(\xi+\eta)\cdot x}\, d\xi d\eta dt}{L^p_{\lan{x}^a}}\lesssim  \n{S_{k-3} f}{L^{p_1}_{\lan{x}^{a_1}}} \n{\De_k g}{L^{p_2}_{\lan{x}^{a_2}}},
\end{equation}
where the implicit constant is independent of $k\in \nn$.
\end{lem}

\begin{proof}
The proof of this lemma will mostly follow the same computation as in the proof of Lemma~\ref{le:main}, so we will highlight only the portions which deviates from the original proof.

Considering the Fourier support due to the projection  operators $S_{k-3}$ and $\Delta_k$, we know that the integrand is supported within the region where $|t\xi + \eta| \leq 2^{k+3}R$ for $t\in [0,1]$.  Hence, we can replace $\si_k$ by its Fourier series expansion. Then, we can write the integral inside the norm on the LHS of the desired inequality as
\[
 \sum_{\mm\in \zz^n} c_{\mm}^k \Delta_k g\pr{x - \f{2\pi}{R}2^{-k} \mm } \int_0^1 S_{k-3} f \pr{x - \f{2\pi}{R}2^{-k}t \mm} \, dt.
\]
Let $\bar{p} = \min (p, 1)$.  Then, taking the $L^p_{\lan{x}^a}$ norm, raising to the $\bar{p}$~th power and applying H\"older's inequality, the LHS of \eqref{3.15} is bounded by
\[
\sum_{\mm \in \zz^n} |c_\mm^k|^{\bar{p}}\n{ \int_0^1 S_{k-3} f \pr{\cdot - \f{2\pi}{R}2^{-k}t \mm} \, dt}{L^{p_1}_{\lan{x}^{a_1}}}^{\bar{p}} \n{\Delta_k g \pr{\cdot - \f{2\pi}{R}2^{-k}t \mm}}{L^{p_2}_{\lan{x}^{a_2}}}^{\bar{p}}.
\]
Since $p_1\geq 1$, we can apply Minkowski's inequality to pull the integral sign out of the $L^{p_1}_{\lan{x}^{a_1}}$ norm.  Applying  \eqref{eq:cmk} to bound the Fourier coefficients, the expression above is bounded by
\[
\sum_{\mm \in \zz^n}  \pr{\lan{\mm}^{-M\bar{p}}\lan{2^{-k}\mm}^{\f{a\bar{p}}{p}}  +  \lan{\mm}^{-\ga \bar{p}} e^{- c'\bar{p} 2^{-k}|\mm|}\lan{2^{-k}\mm}^{\f{a\bar{p}}{p}} }  \lan{ 2^{-k} \mm  }^{\f{a \, \bar{p}}{p}}  \n{ S_{k-3} f }{L^{p_1}_{\lan{x}^{a_1}}}^{\bar{p}} \n{g }{L^{p_2}_{\lan{x}^{a_2}}}^{\bar{p}},
\]
which leads to the desired estimate as in the proof of Lemma~\ref{le:main}.
\end{proof}

Although Lemma~\ref{le:commutator} requires the same decay estimate as Lemma~\ref{le:main}, we will see that, when applying this lemma, the $\{\si_k\}_{k\in \nn}$ will always decay rapidly.  In other words, this lemma is not as important in regards to obtaining the sharp range of $s$ as Lemma~\ref{le:main}.  Rather, the main contribution of this lemma is to enable the low-to-high frequency exchange through a commutator estimate, which produces extra summability to help in the $L^1$ and $L^\infty$ endpoint cases.

Finally, we state the interpolation lemma used in \cite{OW} which helps simplify our computations.  Since the only difference here is the addition of polynomial weights, which does not affect the proof, we omit the proof of the following lemma.

\begin{lem}\label{le:2.6A}
 If $\mathfrak{a}_k \lesssim \min \{2^{k a} A, 2^{-k b} B\}$ for some $a,b,A,B>0$ and any $k\in\zz,$
then,  for any $u>0$,  we have $\{\mathfrak{a}_k\}_{k\in\zz} \in \ell^u (\zz)$ and
\begin{equation*}
\n{\{\mathfrak{a}_k\}_{k\in \zz}}{\ell^u} \lesssim_{a,b,u} A^{\frac{b}{a+b}} B^{\frac{a}{a+b}}.
\end{equation*}
In particular, if for each $k\in \zz$, $\|f_k\|_{L^r_w} \lesssim |\mathfrak{a}_k|$ for some $0<r \le \infty$ and some weight function $w$,  then
$$\n{\sum_{k\in \zz} f_k}{L^r_w}  \lesssim_{a,b,u}  A^{\frac{b}{a+b}} B^{\frac{a}{a+b}}.$$
\end{lem}

With these technical lemmas equipped, we are ready to prove Theorem~\ref{main1}.  We divide the proof of this theorem into two parts: Subsection~\ref{sec:inhom2} contains the proof of the positive statement \eqref{eq:1.3}, and Subsection~\ref{sec:inhom3} treats the negative statement which claims the sharpness of the range of $s$.

\subsection{Proof of the Kato-Ponce Inequality in Polynomially Weighted Spaces}\label{sec:inhom2}
We begin with the following paraproduct decomposition:
\[
J^s(fg) = \sum_{k\geq 0} J^s (S_{k-3} f \De_k g) + \sum_{k\geq 0}J^s (\De_k f S_{k-3} g) + \sum_{k\geq 0} J^s ( \De_k f \, \wt{\De_k} g) + J^s (S_0f \, S_0g) =: I_1+ I_2 + I_3 +I_4.
\]
It suffices to establish the following estimates:
\[
\n{I_j}{L^p_{\langle x \rangle^a}} \lesssim   \n{J^s f}{L^{p_1}_{\langle x \rangle^{a_1}}} \n{g}{L^{p_2}_{\langle x \rangle^{a_2}}} + \n{f}{L^{p_1}_{\langle x \rangle^{a_1}}}\n{J^s g}{L^{p_2}_{\langle x \rangle^{a_2}}}
\]
for $j=1,2,3,4$ where all parameters satisfy the conditions stated in Theorem~\ref{main1}.\\

\noindent
\textbf{Estimate for $I_4$:}\\
Denote $h:= S_0f \cdot S_0g$ and note that $\wh{h} =\widehat{S_{3} h}=\widehat{S_{3}^2 h}$.  Then
\begin{align*}
I_4 &=  \int_{\rr^{n+n}} \langle\xi\rangle^{s}\, \Phi(2^{-3} \xi) \,\wh{ S_3 h}(\xi)\, e^{i \xi\cdot x}\, d\xi.
 \end{align*}
Applying Lemma~\ref{le:main} with $\si_k = \si_1 (\cdot) =  \langle\cdot\rangle^{s} \Phi(2^{-3} \cdot) \in C^{\infty}_c (\rr^n)$, H\"{o}lder's inequality and Lemma \ref{Young}, we obtain
\[
\n{I_4}{L^p_{\langle x \rangle^a}} \lesssim \|S_3h\|_{L^p_{\langle x \rangle^a}} \lesssim \n{S_0 f}{L^{p_1}_{\langle x \rangle^{a_1}}} \n{S_0 g}{L^{p_2}_{\langle x \rangle^{a_2}}}. \lesssim  \n{ S_0  f}{L^{p_1}_{\langle x \rangle^{a_1}}} \n{g}{L^{p_2}_{\langle x \rangle^{a_2}}}.
\]
Similar arguments lead to the boundedness of $J^{-s}S_0$ on $L^{p_1}_{\lan{x}^{a_1}}$, which further implies
\begin{equation*}
\n{  S_0 f}{L^{p_1}_{\langle x \rangle^{a_1}}}= \n{ J^{-s} S_0 J^s f}{L^{p_1}_{\langle x \rangle^{a_1}}} \lesssim \n{J^s f}{L^{p_1}_{\lan{x}^{a_1}}}.
\end{equation*}
Combining the above estimates gives the desired bound for $I_4$. \\

\noindent
\textbf{Estimate for $I_3$}:\\
Note that the Fourier transform of this summand is supported within a ball of radius $2^{k + 5}$. Thus, denoting $h := \De_k f \, \wt{\De_k}g$, we can write $\wh h = \wh{S^2_{k+5} h}$ which leads to
\begin{align*}
J^s ( \De_k f \, \wt{\De_k}g)
 &= 2^{ks} \int_{\rr^n} (2^{-2k}+|2^{-k}\xi|^2)^{\f{s}{2}} \,\Phi(2^{-k-5}\xi) \,\wh{ S_{k+5} h}(\xi)  e^{i\xi\cdot x}\, d\xi.
 \end{align*}
We define $\si_k(\cdot) := (2^{-2k}+|\cdot|^2)^{\f{s}{2}} \,\Phi(2^{-5}\cdot)$ in this case, which satisfies by Lemma~\ref{le:jsde}
\[
|\wh{\si_k} (x) | = J^{s}_{2^{-k}} \wh{\Phi} (x) \lesssim_M \lan{x}^{-M} + \lan{x}^{-n-s} e^{-c 2^{-k}|x|}.
\]
Alternatively, if $s \in 2\nn$, we know that $\wh{\si_k}$ decays like a Schwartz function.  Applying Lemma~\ref{le:main} with $\ga = n+s$, and noting that $n+s > \f{n}{p}$ from the assumptions in Theorem~\ref{main1}, we have
\begin{align*}
\|J^s ( \De_k f \, \wt{\De_k}g) \|_{L^p_{\langle x \rangle^a}} &= 2^{ks} \n{\int_{\rr^n} \si_k(2^{-k}\xi) \,\wh{ S_{k+5} h}(\xi)  e^{i\xi\cdot x}\, d\xi}{L^p_{\langle x \rangle^a}}\\
&\lesssim 2^{ks}\n{\De_k f \, \wt{\De_k}g }{L^p_{\langle x \rangle^a}}\\
&\lesssim 2^{ks} \n{\De_k f}{L^{p_1}_{\langle x \rangle^{a_1}}} \|\wt{\De_k}g\|_{L^{p_2}_{\langle x \rangle^{a_2}}}.
\end{align*}
We close the argument by applying Lemma~\ref{le:2.6A}.  First, define
\[
\mathfrak{a}_k := 2^{ks} \n{\De_k f}{L^{p_1}_{\langle x \rangle^{a_1}}} \n{\wt{\De_k} g}{L^{p_2}_{\langle x \rangle^{a_2}}} \qquad \tn{ for } k\in \nn.
\]
By Lemmas \ref{lp} and \ref{Bern} (ii), we get
\[
\mathfrak{a}_{k} \lesssim \min \left\{ 2^{ks} \n{f}{L^{p_1}_{\langle x \rangle^{a_1}}} \,  \n{g}{L^{p_2}_{\langle x \rangle^{a_2}}},\ 2^{-ks} \n{J^s f}{L^{p_1}_{\langle x \rangle^{a_1}}} \,  \n{J^s g}{L^{p_2}_{\langle x \rangle^{a_2}}} \right\}.
\]
Applying Lemma~\ref{le:2.6A} with
\[
a = b  = s, \quad A = \n{f}{L^{p_1}_{\langle x \rangle^{a_1}}} \,  \n{g}{L^{p_2}_{\langle x \rangle^{a_2}}}, \quad B = \n{J^s f}{L^{p_1}_{\langle x \rangle^{a_1}}} \,  \n{J^s g}{L^{p_2}_{\langle x \rangle^{a_2}}},
\]
we have
\begin{align*}
\n{\sum_{k\geq 0} J^s ( \De_k f \, \wt{\De_k}g) }{L^p_{\langle x \rangle^a}} &\lesssim \pr{\n{f}{L^{p_1}_{\langle x \rangle^{a_1}}} \,  \n{g}{L^{p_2}_{\langle x \rangle^{a_2}}}}^{\f{1}{2}} \pr{\n{J^s f}{L^{p_1}_{\langle x \rangle^{a_1}}} \,  \n{J^s g}{L^{p_2}_{\langle x \rangle^{a_2}}}  }^{\f{1}{2}}\\
&\leq  \n{J^s f}{L^{p_1}_{\langle x \rangle^{a_1}}} \n{g}{L^{p_2}_{\langle x \rangle^{a_2}}} + \n{f}{L^{p_1}_{\langle x \rangle^{a_1}}}\n{J^s g}{L^{p_2}_{\langle x \rangle^{a_2}}}.
\end{align*}
This establishes the desired estimate for $I_3$.\\

\noindent
\textbf{Estimates for $I_1$ and $I_2$:}\\
Since $I_1$ and $I_2$ are symmetric, it suffices to show the desired estimate for $I_1$.  This term is not directly summable if using the method of estimating $I_3$, but it becomes summable after the low-to-high frequency exchange as introduced in \cite{BL}.  We begin by writing $I_1$ as follows:
\[
I_1 = \sum_{k\geq 0} [J^s, S_{k-3} f] \De_k g + f J^s \De_{>0} g - \sum_{k\geq 0} \De_{>k-3} f J^s \De_{k} g =: I_{1.1} + f J^s \De_{>0} g - I_{1.2},
\]
where $[J^s, S_{k-3}f]$ is a commutator.  The middle term $f J^s\Delta_{>0}g$ is estimated by the H\"older's inequality when observing that $\De_{>0} = I - S_0$ is bounded in $L^{p_2}_{\lan{x}^{a_{2}}}$.

$I_{1.2}$ can be treated in the same way as $I_3$. Indeed, by H\"{o}lder's inequality as well as Lemma \ref{Bern} (i) and (ii),
$$
\|\Delta_{>k-3} f J^s \Delta_k g\|_{L^p_{\langle x\rangle^a}}  \lesssim \n{\De_{>k-3} f}{L^{p_1}_{\langle x \rangle^{a_1}}} \,  \n{J^s \De_k g}{L^{p_2}_{\langle x \rangle^{a_2}}}  \lesssim  \min\left\{2^{ks} \n{f}{L^{p_1}_{\langle x \rangle^{a_1}}}\n{g}{L^{p_2}_{\langle x \rangle^{a_2}}}, 2^{-ks} \|J^s f\|_{L^{p_1}_{\langle x \rangle^{a_1}}} \|J^s g\|_{L^{p_2}_{\langle x \rangle^{a_2}}}\right\},
$$
which leads to the desired estimate for $I_{1.2}$ after applying Lemma \ref{le:2.6A} as before.

For $I_{1.1}$, we will need the following commutator estimate:
\begin{prop}\label{prop:key2}
Let $a_1,a_2\ge 0$, $p\in [\f{1}{2}, 1]$ and $p_1,p_2\in [1,\infty]$ satisfying $ \f{1}{p_1} + \f{1}{p_2}=\f{1}{p}$ and $\f{a_1}{p_1} + \f{a_2}{p_2} = \f{a}{p}$.  Then, for any $s > 0$ and $k\geq 0$,
\begin{equation*}
\n{[J^s, S_{k-3}f] \De_k g}{L^p_{\langle x \rangle^a}} \lesssim   2^{k(s-1)}  \n{\na S_{k-3} f}{L^{p_1}_{\langle x \rangle^{a_1}}} \n{\De_k g}{L^{p_2}_{\langle x \rangle^{a_2}}}.
\end{equation*}
\end{prop}

\begin{proof}
We begin by writing
\begin{align*}
[J^s, S_{k-3}f] \De_k g
&=\int_{\rr^{n+n}} \pr{\langle\xi+\eta\rangle^s - \langle\eta\rangle^s}  \wh{S_{k-3} f}(\xi)  \wh{\De_k g}(\eta) e^{i(\xi+\eta)\cdot x}\, d\xi \, d\eta\\
&=\int_{\rr^{n+n}} \pr{\int_0^1s \xi \cdot (t\xi+\eta) \langle t\xi + \eta\rangle^{s-2} dt}  \wh{S_{k-3} f}(\xi) \wh{\De_k g}(\eta) e^{i(\xi+\eta)\cdot x}\, d\xi \, d\eta.
\end{align*}
Noticing that the support of the integrand forces
$
\langle t\xi+\eta\rangle\sim |\eta|\sim 2^k
$
uniformly in $t\in [0,1],$
we may choose $\psi\in \cs(\rr^n)$ supported on an annulus so that $[J^s, S_{k-3}f] \De_k g$ can be written as (up to a constant)
\begin{align*}
\int_0^1\int_{\rr^{n+n}} \xi \cdot &(t\xi+\eta)\langle  t\xi+\eta\rangle^{s-2} \psi(2^{-k} (t\xi+\eta)) \wh{S_{k-3} f}(\xi) \wh{\De_k g}(\eta) e^{i(\xi+\eta)\cdot x}\, d\xi \, d\eta\, dt\\
&=
  2^{k(s-1)} \int_0^1 \int_{\rr^{n+n}}  \si_k(2^{-k} (t\xi+\eta))\cdot  \wh{\na S_{k-3} f}(\xi) \wh{\De_k g}(\eta) e^{i(\xi+\eta)\cdot x}\, d\xi \, d\eta\, dt,
\end{align*}
where $\{\si_k\}_{k\in \bn} :=\{(\cdot) (2^{-2k}+|\cdot|^2)^{\f{s-2}{2}} \psi(\cdot)\}_{k\in \bn}$ is a family of compactly supported $C^\infty$ functions.  Direct computations show that the partial derivatives of $\si_k$ is uniformly bounded in $k\in \nn$.  Using arguments from the proof of Lemma~\ref{Bern}(i), we can see that $\abs{\wh{\si_k}(x)} \lesssim_M \lan{x}^{-M}$ for any $M \gg 1$ where the implicit constant is independent of $k\in \nn$.  Hence, we can apply Lemma~\ref{le:commutator} to obtain the desired estimate.
\end{proof}

We continue to estimate $I_{1.1}$. Lemma~\ref{Bern} (iv) gives us that, for $0<\ve\leq \min(1,s)$,
\begin{align*}
\begin{split}
\| \nabla S_{k-3}f \|_{L^{p_1}_{\langle x \rangle^{a_1}}}&
\lesssim 2^{k(1-\ve)} \|f\|_{L^{p_1}_{\langle x \rangle^{a_1}}}^{\frac{s-\ve}{s}} \|J^s f\|_{L^{p_1}_{\langle x \rangle^{a_1}}}^{\frac{\ve}{s}}.
\end{split}
\end{align*}
From this and Proposition \ref{prop:key2}, it follows that
\begin{equation*}
\n{[J^s, S_{k-3}f] \De_k g}{L^p_{\langle x \rangle^a}} \lesssim  \min\left\{  2^{ks}  \n{f}{L^{p_1}_{\langle x \rangle^{a_1}}} \n{g}{L^{p_2}_{\langle x \rangle^{a_2}}},  \, 2^{-k\ve}  \n{f}{L^{p_1}_{\langle x \rangle^{a_1}}}^{\frac{s-\ve}{s}}\n{J^sf}{L^{p_2}_{\langle x \rangle^{a_1}}}^{\frac{\ve}{s}} \n{J^s g}{L^{p_2}_{\langle x \rangle^{a_2}}} \right\}.
\end{equation*}
Applying Lemma~\ref{le:2.6A} with $a=s$, $b=\ve$ as well as
\[
 A = \n{f}{L^{p_1}_{\langle x \rangle^{a_1}}} \,  \n{g}{L^{p_2}_{\langle x \rangle^{a_2}}}, \quad B = \n{f}{L^{p_1}_{\langle x \rangle^{a_1}}}^{\frac{s-\ve}{s}}\n{J^sf}{L^{p_2}_{\langle x \rangle^{a_1}}}^{\frac{\ve}{s}} \n{J^s g}{L^{p_2}_{\langle x \rangle^{a_2}}}
\]
leads to
\begin{align*}
\|I_{1.1}\|_{L^p_{\langle x \rangle^a}} &\lesssim \pr{\n{f}{L^{p_1}_{\langle x \rangle^{a_1}}} \n{ g}{L^{p_2}_{\langle x \rangle^{a_2}}} }^{\frac{\ve}{s+\ve}}  \pr{ \n{J^s  g}{L^{p_2}_{\langle x \rangle^{a_2}}}  \n{f}{L^{p_1}_{\langle x \rangle^{a_1}}}^{\frac{s-\ve}{s}}\n{J^sf}{L^{p_1}_{\langle x \rangle^{a_1}}}^{\frac{\ve}{s}}}^{\f{s}{s+\ve}}\\
&= \pr{\|J^s g\|_{L^{p_2}_{ \langle x\rangle^{a_2}}} \|f\|_{L^{p_1}_{\langle x \rangle^{a_1}}}}^{\frac{s}{s+\ve}} \pr{\|J^s f\|_{L^{p_1}_{\langle x \rangle^{a_1}}}
\|g\|_{L^{p_2}_{\langle x \rangle^{a_2}}}}^{\frac{\ve}{s+\ve}} \\
&  \lesssim \|f\|_{L^{p_1}_{\langle x \rangle^{a_1}}} \|J^s g\|_{L^{p_2}_{\langle x \rangle^{a_2}}}+ \|J^s f\|_{L^{p_1}_{\langle x \rangle^{a_1}}} \|g\|_{L^{p_2}_{\langle x \rangle^{a_2}}}.
\end{align*}
This leads to the desired estimate for $I_1$.

Putting together the estimates obtained above for $I_1,I_2,I_3$ and $I_4$, we conclude the proof of \eqref{eq:1.3}.

\subsection{Proof of the Sharpness of the Range of $s$}\label{sec:inhom3}

In this section, we want to show that the range of $s$ in Theorem~\ref{main1} is sharp by deriving a contradiction to \eqref{eq:1.3} when  $\ds s \leq  \max\left\{n \pr{\f{1}{p} -1},0\right\}$ and $s \not\in 2\nn$.

For $s< 0$, we will verify that the counterexample from \cite{GO} still works in the weighted setting.  For $k\gg 1,$ let
\[
f(x)=f_k(x)=e^{i2^k e_1\cdot x} \wh{\Phi}(x),\qquad g(x)=g_k(x)=e^{-i2^k e_1\cdot x} \wh{\Phi}(x),
\]
where $e_1=(1,0,\ldots,0)\in \rr^n$ and $\Phi,\Psi\in \S(\rr^n)$ are as defined in Section~\ref{sec:prelim}.  From  \eqref{eq:1.3}, only two terms explicitly depend on $k$: namely, $\n{J^s f}{L^{p_1}_{\lan{x}^{a_1}}}$ and  $\n{J^s g}{L^{p_2}_{\lan{x}^{a_2}}}$.  We can write
\begin{align*}
    J^s f(x) & =\int_{\rr^n} \lan{\xi}^s \Phi(\xi-2^k e_1) e^{i\xi\cdot x}d\xi\\
    & =\int_{\rr^n}  2^{ks} \Psi^s_k( 2^{-k}\xi) \Phi(\xi-2^k e_1) e^{i\xi\cdot x}d\xi,
\end{align*}
where $\Psi^s_k(\xi)=(2^{-2k}+|\xi|^2)^{\frac{s}{2}}\Psi(\xi)$ is as defined during the proof of Lemma~\ref{Bern}(i).  By Lemma~\ref{Young},
\[
\|J^s f\|_{L^{p_1}_{\lan{x}^{a_1}}}\lesssim 2^{ks} \n{2^{kn} \wh{\Psi^s_k}(2^k\cdot)}{L^1_{\lan{x}^{a_1}}} \n{\wh{\Phi}}{L^{p_1}_{\lan{x}^{a_1}}}.
\]
We have shown in \eqref{eq:psisj} that the first norm on the RHS above is uniformly bounded in $k$.  Since  $\n{J^s g}{L^{p_2}_{\lan{x}^{a_2}}}$ can be treated analogously, we can see that the RHS of \eqref{eq:1.3} approaches zero as $k\to \infty$ when $s<0$, while the LHS remains a positive constant.  This leads to a contradiction.

Next, we assume $0< s\leq n\pr{\f{1}{p} - 1}$ and $s\not\in 2\nn$.   In this case, we will need the exponential decay estimate derived in Lemma~\ref{le:jsde} to extend the counterexample from \cite{GO} to weighted setting.

Let $f,g \in C^{\infty}_c (\rr^n)$ be nonzero nonnegative radial functions.  For any $\de \in (0,1]$, we define $f_\de (x) = f\pr{\f{x}{\de}}$ and similarly $g_\de(x) = g\pr{\f{x}{\de}}$.
Assuming \eqref{eq:1.3} is true, we must have
\begin{equation} \label{sharp}
\n{J^s (f_\de g_\de)}{L^p_{ \langle x\rangle^a
}} \lesssim  \n{J^s f_\de}{L^{p_1}_{\langle x \rangle^{a_1}}} \n{g_\de}{L^{p_2}_{\langle x \rangle^{a_2}}} + \n{f_\de}{L^{p_1}_{\langle x \rangle^{a_1}}}\n{J^s g_\de}{L^{p_2}_{\langle x \rangle^{a_2}}},
\end{equation}
which is equivalent to
\begin{equation} \label{sharp3}
\n{J^s_\de (f g)}{L^p_{ \langle \de x \rangle^a
}} \lesssim  \n{J^s_\de f}{L^{p_1}_{\langle \de x \rangle^{a_1}}} \n{g}{L^{p_2}_{\langle \de  x \rangle^{a_2}}} + \n{f}{L^{p_1}_{\langle \de  x \rangle^{a_1}}}\n{J^s_\de g}{L^{p_2}_{\langle \de  x\rangle^{a_2}}},
\end{equation}
where the implicit constant is independent of $\de>0$.

First, we will show that the RHS of \eqref{sharp3} is uniformly bounded in $\delta\in (0,1]$.
We begin by finding an upper bound for $\n{J_\de^s f}{L^{p_1}_{\langle \de x \rangle^{a_1}}}$. Applying Lemma~\ref{le:jsde}, we can write
\begin{align*}
\n{J_\de^s f}{L^{p_1}_{\langle \de x \rangle^{a_1}}}^{p_1} &\lesssim  \int_{\rr^n}  \lan{x}^{-p_1 M}\lan{\de x}^{a_1} + \lan{x}^{-(n+s)p_1} e^{-cp_1 \de |x|}\lan{\de x}^{a_1} \, dx,\quad \forall M\gg 1.
\end{align*}
Taking $M$ sufficiently large and noticing that $\n{e^{-c p_1 (\cdot)} \lan{\cdot}^{a_1}}{L^{\infty}} \lesssim 1$ and $(n+s)p_1 \geq n+s >n$, we see that the integral on the RHS above is uniformly bounded in $\delta\in (0,1]$.
An upper bound for $\n{g}{L^{p_2}_{\lan{\de x}^{a_2}}}$ is much easier to find since $\n{g}{L^{p_2}_{\langle \de x \rangle^{a_2}}} \lesssim  \n{g}{L^{p_2}_{\langle  x \rangle^{a_2}}}$.  By symmetry, the second term on the RHS of \eqref{sharp3} is uniformly bounded in $\de \in (0,1]$.  Putting together the estimates obtained above, the RHS of \eqref{sharp3} is uniformly bounded in $\de \in (0,1]$.

Next, we will show that the LHS of \eqref{sharp3} is not uniformly bounded in $\delta\in (0,1]$. Applying the Lebesgue dominated convergence, we can show that $\abs{J^s_\de(fg) (x)}^p \lan{\de x}^a \to \abs{D^s(fg)(x)}^p$ pointwise as $\de \to 0$.
Then, Fatou's lemma tells us that
\begin{equation} \label{eq:dsfg}
\liminf_{\de\to 0}  \int_{\rr^n} |J^s_\de(fg)(x)|^{p} \lan{\de x}^{a}\,dx \geq  \int_{\rr^n} |D^s(fg)(x)|^{p} \,dx.
\end{equation}
By our choice of $f$ and $g$, Lemma~\ref{le:GO} implies $|D^s(fg)|(x) \gtrsim |x|^{-n-s}$ for $|x|\gg 1$.    Hence, the integrand on the RHS of \eqref{eq:dsfg} is bounded from below by $|x|^{-(n+s)p} \chi_{|x|\gg 1}(x)$, which implies the divergence of the integral on the RHS of \eqref{eq:dsfg} since
\[
0 < s \leq n\pr{\f{1}{p} - 1} \implies -(n+s)p \geq -n.
\]
 This tells us that the LHS of \eqref{sharp3} cannot be uniformly bounded in $\de \in (0,1]$, which leads to a contradiction.  Therefore, \eqref{eq:1.3} cannot be true for $0<s \leq n\pr{\f{1}{p} - 1}$ and $s\not\in 2\nn$.  This proves the sharpness of the range of $s$ in Theorem~\ref{main1}.

\section{Commutator Estimates with Polynomial Weights}\label{sec:com}
While a large part of the proof of the Kato-Ponce commutator estimates is similar to that of the Kato-Ponce inequality in the previous section, a few components of the proof are significantly different.  This is somewhat expected considering that a counterexample to the $L^\infty$-endpoint given in \cite{BL} necessitates a more complicated form of the inequality, namely \eqref{eq:new1.5}, when $s>1$ and $s\not\in 2\nn$.

The proof of Theorem~\ref{main-comm} is organized similarly as in the previous section.  We begin by stating variants of several technical lemmas in Subsection~\ref{sec:inhom1}.

\subsection{Lemmas}\label{sec:com1}
 In \cite{BL}, the authors used the following interpolation inequality to prove the $L^\infty$-endpoint commutator estimate:
\begin{align*}
\|D^{\theta s} f\|_{L^r} \le \|D^{s} f\|_{L^p}^{\theta} \| f\|_{L^q}^{1-\theta},
\end{align*}
where $s\ge 0$, $0<\theta<1$, $1<p,q<\infty$, $1< r \leq \infty$ satisfy
$ \frac{1}{r}=\frac{\theta}{p}+\frac{1-\theta}{q}$.  However, this interpolation inequality fails when $r=1$, which we will need to obtain the $L^1$~endpoint commutator estimate.  For $0<s<1$, the following localized version can replace this interpolation inequality.

\begin{lem}\label{lem:inter}
Let $0\leq s < 1$, $1\le r\le \infty$, $a\geq 0,$ $ 0< \theta \leq 1$, and $k\in \nn$. For any $f\in \S(\rr^n)$,
\begin{align*}
\| S_k f\|_{L^r_{\langle x\rangle^a}} &\lesssim   2^{k\theta(1-s)}  \|J^{s-1} f\|_{L^r_{\langle x\rangle^a}}^{\theta} \| f\|_{L^r_{\langle x\rangle^a}}^{1-\theta},\\
\| \nabla S_k f\|_{L^r_{\langle x\rangle^a}} &\lesssim   2^{k\theta(1-s)}  \|J^{s} f\|_{L^r_{\langle x\rangle^a}}^{\theta} \| \nabla f\|_{L^r_{\langle x\rangle^a}}^{1-\theta}.
\end{align*}
\end{lem}

\begin{proof}
The first inequality follows immediately from Lemmas~\ref{Bern}(iii) and \ref{Young}:
\[
  \|S_k f\|_{L^r_{\langle x\rangle^a}} =  \|J^{1-s} S_k J^{s-1} f\|_{L^r_{\langle x\rangle^a}}^{\theta}   \|S_k f\|_{L^r_{\langle x\rangle^a}}^{1-\theta} \lesssim 2^{k\theta(1-s)}  \n{J^{s-1} f}{L^r_{\langle x\rangle^a}}^{\theta} \n{ f}{L^r_{\langle x\rangle^a}}^{1-\theta}.
\]

The second inequality appears the same as the first one when we replace $f$ by $\na f$.  However, in doing so, we end up with $J^{s-1} \na f$ on the RHS rather than $J^s f$.  This can be separately treated by going through the proof of Lemma~\ref{Bern}(iii).  As in that proof, we can write
\begin{equation}\label{eq:le4}
\n{\na S_k f}{L^r_{\langle x\rangle^a}} \lesssim \n{\na S_0 f}{L^r_{\langle x\rangle^a}} + \sum_{j=0}^k \n{\na \De_j f}{L^r_{\langle x\rangle^a}}.
\end{equation}
By Lemma~\ref{Young}, we know that both $S_0$ and $J^{-s}\na S_0$ are bounded on $L^r_{\lan{x}^a}$, which gives
\[
\n{\na S_0 f}{L^r_{\langle x\rangle^a}} = \n{J^{-s}\na S_0 J^s f}{L^r_{\langle x\rangle^a}}^{\theta}\n{S_0 \na f}{L^r_{\langle x\rangle^a}}^{1-\theta} \lesssim \n{J^s f}{L^r_{\langle x\rangle^a}}^{\theta}\n{ \na f}{L^r_{\langle x\rangle^a}}^{1-\theta}.
\]
From \eqref{eq:conv}, Lemmas~\ref{Bern} (ii) and \ref{lp}, we get that, for $j\in \nn$,
\[
\n{\na \De_j f}{L^r_{\langle x\rangle^a}} \lesssim 2^{j\theta} \n{\De_j f}{L^r_{\langle x\rangle^a}}^{\theta} \n{\De_j \na f}{L^r_{\langle x\rangle^a}}^{1-\theta} \lesssim 2^{j\theta(1-s)}  \n{ J^s f}{L^r_{\langle x\rangle^a}}^{\theta} \n{\na f}{L^r_{\langle x\rangle^a}}^{1-\theta},
\]
where the implicit constant is independent of $j.$
Since $\theta(1-s)>0$ by our hypothesis, substituting these bounds to \eqref{eq:le4} yields the second inequality.
\end{proof}

For $s<1$, the lemma above, along with other lemmas already introduced in Subsection~\ref{sec:inhom1}, is sufficient to prove the commutator estimate, as we will see in Subsection~\ref{sec:com2}.  For $s>1$, we will need additional estimates resembling the ones given in Lemma~\ref{Bern}, mainly to handle the local differential operator $\na$.

\begin{lem}\label{lem:4.2}
Let $s>1$, $p\in [1,\infty]$, $k \in \nn$ and $a\geq 0$.
\begin{enumerate}
\item[(i)] $\ds   \n{\De_k f}{L^{p}_{\lan{x}^{a}}} \sim 2^{-k}\n{\na \De_k f}{L^{p}_{\lan{x}^{a}}}.$

\item[(ii)] $\ds    \n{\na \De_{>k} f}{L^{p}_{\lan{x}^{a}}} \lesssim 2^{k(1-s)}\n{J^s f}{L^{p}_{\lan{x}^{a}}}.$

\item[(iii)] For any $0\leq \ve< \min(1,s-1)$,
\[
\n{|\na|^2 S_{k-3} f }{L^{p_1}_{\langle x \rangle^{a_1}}}  \lesssim 2^{k(1-\ve)} \n{\na f }{L^{p_1}_{\langle x \rangle^{a_1}}}^{\f{s-1-\ve}{s-1}}\n{J^{s} f }{L^{p_1}_{\langle x \rangle^{a_1}}}^{\f{\ve}{s-1}}.
\]
\end{enumerate}
\end{lem}


\begin{proof}

Inequality (i) has two directions, one of which was already proved in \eqref{eq:conv}.  To prove the converse direction, write
\begin{align*}
\De_k f(x) &= \int_{\rr^n} \Psi(2^{-k} \xi)\wh{f}(\xi)e^{i\xi\cdot x}\, d\xi\\
	&= \int_{\rr^n} \pr{\xi|\xi|^{-2}\widetilde{\Psi}(2^{-k} \xi)} \cdot \pr{\xi \Psi(2^{-k} \xi)\wh{f}(\xi)}e^{i\xi\cdot x}\, d\xi\\
	&= 2^{-k} \left[{2^{kn}\wh{\psi}}(2^k\cdot) *  \na \De_k f\right] (x),
\end{align*}
where {$\widetilde{\Psi}$ is a radial Schwartz function supported in a larger annulus than $\Psi$ such that $\Psi\widetilde{\Psi}=\Psi,$} and $\psi (\cdot) :=( \cdot) |\cdot|^{-2} \widetilde{\Psi}(\cdot)$.  Since $\widetilde{\Psi}$ is supported in an annulus, we know $\psi \in \cs(\rr^n)$.  Applying the weighted Young inequality given in Lemma~\ref{Young}, we obtain
\[
\n{\De_k f}{L^{p}_{\lan{x}^{a}}} \lesssim 2^{-k}\n{\na \De_k f}{L^{p}_{\lan{x}^{a}}}.
\]
This establishes the converse inequality.

Inequality (ii) follows from the proof of Lemma~\ref{Bern}(i), when we apply statement (i) of the current lemma.

Also, (iii) easily follows when we apply Lemma~\ref{Bern}{(iv)} to $\na f$ instead of $f$ {with $s$ replaced by $s-1$}.  A direct application of this statement leads to the RHS of the resulting inequality to contain $\na J^{s-1} f$ rather than $J^sf$.  However, within the proof given in Lemma~\ref{Bern}{(iv)}, we can apply {the $L^p_{\lan{x}^a}$ boundedness of $\na J^{-1}S_0$ and $\na J^{-1}\De_j$ (with constants independent of $j\in \bn$), which is an easy consequence of the weighted Young inequality.}  The details are omitted.
\end{proof}

We now proceed to prove Theorem~\ref{main-comm}, which contains two different versions of the Kato-Ponce commutator estimates: \eqref{eq:new1.4} and \eqref{eq:new1.5}.  We can decompose $[J^s, f] g$ as in Subsection~\ref{sec:inhom2}:
\begin{equation}\label{eq:decomp}
\sum_{k\geq 0} [J^s, S_{k-3} f] \De_k g + \sum_{k\geq 0} [J^s, \De_k f] S_{k-3} g + \sum_{k\geq 0} [J^s , \De_k f]  \wt{\De_k}g + [J^s, S_0f]  S_0g=: II_1 + II_2 + II_3 + II_4.
\end{equation}
Note that the RHS of both inequalities are the same, while the LHS of \eqref{eq:new1.5} contains an extra term, $s\na f\cdot \na J^{s-2} g$.  This term is needed to achieve some cancellation when estimating $II_1$ only when $s>1$.

The proof of Theorem~\ref{main-comm} is organized as follows.
Estimates for $II_3$ and $II_4$ are the same for both \eqref{eq:new1.4} and \eqref{eq:new1.5}, so these will be presented first in Subsection~\ref{sec:com2}.  When $0<s<1$, estimates for $II_1$ and $II_2$ are similar to the ones given in Subsection~\ref{sec:inhom2}.  However, the same method does not work when $s>1$, as evidenced by a different form of this inequality in \eqref{eq:new1.5}.  Hence, we will divide the proof for $II_1$ and $II_2$ based on the range of $s$: Subsection~\ref{sec:com2} will contain the proof of estimates for $II_1$ and $II_2$ when $0<s<1$, and Subsection~\ref{sec:com3} will contain the corresponding estimates with the extra term on the LHS of \eqref{eq:new1.5}  when $s>1$.  In other words, Subsection~\ref{sec:com2} contains the full proof of \eqref{eq:new1.4}, and Subsection~\ref{sec:com3} contains the modifications needed to obtain \eqref{eq:new1.5}.  Finally, Subsection~\ref{sec:com4} is devoted to the proof of the negative direction (that is, sharpness of the lower threshold for $s$) of Theorem~\ref{main-comm}.

\subsection{Proof of \eqref{eq:new1.4}}\label{sec:com2}
We will assume that $0<s<1$ in this subsection in addition to other conditions on parameters stated in Theorem~\ref{main-comm}.
As in Section~\ref{sec:inhom2}, it suffices to establish
\[
\n{II_j}{L^p_{\langle x \rangle^a}} \lesssim   \n{J^s f}{L^{p_1}_{\langle x \rangle^{a_1}}} \n{g}{L^{p_2}_{\langle x \rangle^{a_2}}} + \n{\na f}{L^{p_1}_{\langle x \rangle^{a_1}}}\n{J^{s-1} g}{L^{p_2}_{\langle x \rangle^{a_2}}}
\]
for $j=1,2,3,4$.  \\

\noindent
\textbf{Estimate for $II_4$}:\\ This estimate is similar to that for $I_4$ in Subsection~\ref{sec:inhom2}.   In fact, Lemma~\ref{le:main} implies that the operator $J^{\al}S_0$ is bounded on $L^{p}_{\lan{x}^a}$ for $\al \in \rr$ and $p>0$, which leads to
\[
\n{[J^s, S_0f]  S_0g}{L^p_{\lan{x}^a}} \lesssim \n{J^s(S_0f  S_0g)}{L^p_{\lan{x}^a}}+\n{ S_0f  J^sS_0g}{L^p_{\lan{x}^a}}\lesssim \n{J^s f}{L^{p_1}_{\lan{x}^{a_1}}} \n{g}{L^{p_2}_{\lan{x}^{a_2}}}
\]
by the computations given for $I_4$ in Subsection~\ref{sec:inhom2}.  We omit the details.\\

\noindent
\textbf{Estimate for $II_3$}:\\
This estimate also follows from the same scheme as used for $I_3$ in Subsection~\ref{sec:inhom2} with minor modifications as given below.  We can apply the quasi-triangle inequality to write
\begin{align*}
\n{ [J^s , \De_k f]  \wt{\De_k}g}{L^p_{\lan{x}^a}}\lesssim
 \n{J^s  (\De_k f  \wt{\De_k}g)}{L^p_{\lan{x}^a}}+\n{ \De_k f\, J^s    \wt{\De_k}g}{L^p_{\lan{x}^a}} \lesssim 2^{ks} \n{\De_k f}{L^{p_1}_{\lan{x}^{a_2}}}  \n{\wt{\De_k}g}{L^{p_2}_{\lan{x}^{a_2}}},
\end{align*}
where we applied Theorem~\ref{main1}, H\"older's inequality and Lemmas \ref{Bern}(ii) in the last inequality above. Then, applying Lemmas~\ref{lp}, \ref{Bern}(ii) and \ref{lem:4.2}(i), we can deduce
\begin{equation}\label{eq:ii3}
\n{ [J^s , \De_k f]  \wt{\De_k}g}{L^p_{\lan{x}^a}}\lesssim \min \left\{2^{k(s-1)} \n{\nabla f}{L^{p_1}_{\langle x\rangle^{a_1}}} \n{ g}{L^{p_2}_{\langle x\rangle^{a_2}}},  2^{-k(s-1)} \n{J^s f}{L^{p_1}_{\langle x\rangle^{a_1}}} \n{ J^{s-1}g}{L^{p_2}_{\langle x\rangle^{a_2}}}\right\}.
\end{equation}
As long as $s\neq 1$, we can apply Lemma~\ref{le:2.6A} with $a = b=|s-1|$ and follow the same computations as for $I_3$ in Subsection~\ref{sec:inhom2} to obtain the desired estimate for $II_3$,

As noted previously, the estimates for $II_3$ and $II_4$ are valid  for both $0<s<1$ and $s>1$,  as long as $s > \max\left\{ 0,n\pr{\f{1}{p}-1}\right\}$ or $s\in 2\nn$. Hence, these estimates will not be repeated in Subsection~\ref{sec:com3} when we prove the inequality for $s>1$.\\

\noindent\textbf{Estimate for $II_1$:}\\
Applying Proposition \ref{prop:key2}, we can write
\begin{equation}\label{eq:II1a}
\n{[J^s, S_{k-3}f] \De_k g}{L^p_{\langle x \rangle^a}} \lesssim   2^{k(s-1)}  \n{ \nabla S_{k-3} f}{L^{p_1}_{\langle x \rangle^{a_1}}} \n{\De_k g}{L^{p_2}_{\langle x \rangle^{a_2}}}.
\end{equation}
Since we are assuming $0<s<1$, we can apply and Lemmas~\ref{lp}, \ref{Bern}(ii) and \ref{lem:inter} with $\theta =1$ to the RHS above to write
\[
\n{[J^s, S_{k-3}f] \De_k g}{L^p_{\langle x \rangle^a}} \lesssim   2^{k(1-s)}  \n{ J^s f}{L^{p_1}_{\langle x \rangle^{a_1}}} \n{J^{s-1} g}{L^{p_2}_{\langle x \rangle^{a_2}}}.
\]
On the other hand, applying Lemma~\ref{lem:inter} with $\theta = \f{1}{2}$ to the RHS of \eqref{eq:II1a} yields
\begin{align*}
\n{[J^s, S_{k-3}f] \De_k g}{L^p_{\langle x \rangle^a}} & \lesssim  2^{-\frac{k(1-s)}{2}} \n{ \nabla  f}{L^{p_1}_{\langle x \rangle^{a_1}}}^{\frac{1}{2}} \n{ J^s f}{L^{p_1}_{\langle x \rangle^{a_1}}}^{\frac{1}{2}}
\n{g}{L^{p_2}_{\langle x \rangle^{a_2}}}.
\end{align*}
Hence,
\[
\n{[J^s, S_{k-3}f] \De_k g}{L^p_{\langle x \rangle^a}}\lesssim \min \left\{2^{k(1-s)}  \n{ J^s f}{L^{p_1}_{\langle x \rangle^{a_1}}} \n{J^{s-1} g}{L^{p_2}_{\langle x \rangle^{a_2}}},  2^{-\frac{k(1-s)}{2}} \n{ \nabla  f}{L^{p_1}_{\langle x \rangle^{a_1}}}^{\frac{1}{2}} \n{ J^s f}{L^{p_1}_{\langle x \rangle^{a_1}}}^{\frac{1}{2}}
\n{g}{L^{p_2}_{\langle x \rangle^{a_2}}}\right\}.
\]
Applying Lemma~\ref{le:2.6A} with
\[
a = 1-s,\quad b  = \f{1-s}{2}, \quad A = \n{ J^s f}{L^{p_1}_{\langle x \rangle^{a_1}}} \n{J^{s-1} g}{L^{p_2}_{\langle x \rangle^{a_2}}}, \quad B =  \n{ \nabla  f}{L^{p_1}_{\langle x \rangle^{a_1}}}^{\frac{1}{2}} \n{ J^s f}{L^{p_1}_{\langle x \rangle^{a_1}}}^{\frac{1}{2}}
\n{g}{L^{p_2}_{\langle x \rangle^{a_2}}},
\]
we obtain
\begin{align*}
\n{II_1^A }{L^p_{\langle x \rangle^a}} &\lesssim \pr{\n{ J^s f}{L^{p_1}_{\langle x \rangle^{a_1}}} \n{J^{s-1} g}{L^{p_2}_{\langle x \rangle^{a_2}}}}^{\f{1}{3}} \pr{\n{ \nabla  f}{L^{p_1}_{\langle x \rangle^{a_1}}}^{\frac{1}{2}} \n{ J^s f}{L^{p_1}_{\langle x \rangle^{a_1}}}^{\frac{1}{2}}
\n{g}{L^{p_2}_{\langle x \rangle^{a_2}}} }^{\f{2}{3}}\\
 &= \pr{\n{ J^s f}{L^{p_1}_{\langle x \rangle^{a_1}}}\n{g}{L^{p_2}_{\langle x \rangle^{a_2}}}}^{\f{2}{3}} \pr{ \n{ \nabla  f}{L^{p_1}_{\langle x \rangle^{a_1}}}\n{J^{s-1} g}{L^{p_2}_{\langle x \rangle^{a_2}}}}^{\frac{1}{3}}
\\
&\leq \n{ J^s f}{L^{p_1}_{\langle x \rangle^{a_1}}}\n{g}{L^{p_2}_{\langle x \rangle^{a_2}}} +   \n{ \nabla  f}{L^{p_1}_{\langle x \rangle^{a_1}}}\n{J^{s-1} g}{L^{p_2}_{\langle x \rangle^{a_2}}}.
\end{align*}
This gives us the desired estimate for $II_1$.\\

\noindent\textbf{Estimate for $II_2$:}\\
We begin by writing
\begin{align*}
\sum_{k\geq 0} [J^s , \De_k f]  S_{k-3} g &= \sum_{k\geq 0}J^s  (\De_k f  S_{k-3} g)-  \sum_{k\geq 0} \De_k f\, J^s    S_{k-3}g =: II_{2.1} - II_{2.2}.
\end{align*}

For $II_{2.1}$, we apply Theorem~\ref{main1} as well as Lemma~\ref{Bern}(ii) and (iii) to write
\begin{equation*}
\n{J^s (\De_k f S_{k-3} g)}{L^p_{\langle x\rangle^a}}  \lesssim 2^{ks} \n{\De_k f}{L^{p_1}_{\langle x\rangle^{a_1}}} \n{S_{k-3} g}{L^{p_2}_{\langle x\rangle^{a_2}}}.
\end{equation*}
For any $s>0$, Lemmas~\ref{lp}, \ref{Bern}(ii) and \ref{lem:4.2}(i) give us:
\begin{equation}\label{eq:equiv}
\n{\De_k f}{L^{p_1}_{\langle x \rangle^{a_1}}} \lesssim \min \left\{ 2^{-ks} \n{J^s f}{L^{p_1}_{\langle x \rangle^{a_1}}},  2^{-k} \n{\na f}{L^{p_1}_{\langle x \rangle^{a_1}}}   \right\}.
\end{equation}
For $g$, under the assumption that $s\in (0,1)$, we can apply the first inequality in Lemma~\ref{lem:inter} with $\theta = 1$ and $\f{1}{2}$, respectively, to write
\[
\n{ S_{k-3}g}{L^{p}_{\langle x\rangle^a}} \lesssim \min\left\{2^{k(1-s)}\n{J^{s-1} g}{L^{p}_{\langle x\rangle^a}}, 2^{\f{k}{2}(1-s)}  \n{g}{L^{p_2}_{\langle x\rangle^{a_2}}}^{\f{1}{2}} \n{J^{s-1} g}{L^{p_2}_{\langle x\rangle^{a_2}}}^{\f{1}{2}}\right\}.
\]
Hence,
\begin{equation*}
  \n{J^s (\De_k f S_{k-3} g)}{L^p_{\langle x\rangle^a}}  \lesssim \min \left\{2^{k(1-s)} \n{J^s f}{L^{p_1}_{\langle x\rangle^{a_1}}} \n{J^{s-1} g}{L^{p_2}_{\langle x\rangle^{a_2}}}, 2^{-\f{k}{2}(1-s)} \n{\nabla f}{L^{p_1}_{\langle x\rangle^{a_1}}} \n{g}{L^{p_2}_{\langle x\rangle^{a_2}}}^{\f{1}{2}} \n{J^{s-1} g}{L^{p_2}_{\langle x\rangle^{a_2}}}^{\f{1}{2}}\right\}.
\end{equation*}
Applying Lemma \ref{le:2.6A} with
\[
a = 1-s,\quad b  = \f{1-s}{2}, \quad A = \n{ J^s f}{L^{p_1}_{\langle x \rangle^{a_1}}} \n{J^{s-1} g}{L^{p_2}_{\langle x \rangle^{a_2}}}, \quad B =  \n{ \nabla  f}{L^{p_1}_{\langle x \rangle^{a_1}}} \n{ J^{s-1} g}{L^{p_1}_{\langle x \rangle^{a_1}}}^{\frac{1}{2}}
\n{g}{L^{p_2}_{\langle x \rangle^{a_2}}}^{\frac{1}{2}},
\]
we obtain
\begin{align*}
\n{II_{2.1} }{L^p_{\langle x \rangle^a}} &\lesssim \pr{\n{ J^s f}{L^{p_1}_{\langle x \rangle^{a_1}}} \n{J^{s-1} g}{L^{p_2}_{\langle x \rangle^{a_2}}}}^{\f{1}{3}} \pr{\n{ \nabla  f}{L^{p_1}_{\langle x \rangle^{a_1}}} \n{ J^{s-1} g}{L^{p_1}_{\langle x \rangle^{a_1}}}^{\frac{1}{2}}
\n{g}{L^{p_2}_{\langle x \rangle^{a_2}}}^{\frac{1}{2}} }^{\f{2}{3}}\\
 &= \pr{\n{ J^{s-1} g}{L^{p_1}_{\langle x \rangle^{a_1}}}\n{ \nabla  f}{L^{p_1}_{\langle x \rangle^{a_1}}}}^{\f{2}{3}} \pr{ \n{J^{s} f}{L^{p_2}_{\langle x \rangle^{a_2}}}\n{g}{L^{p_2}_{\langle x \rangle^{a_2}}}}^{\frac{1}{3}}
\\
&\lesssim \n{ J^s f}{L^{p_1}_{\langle x \rangle^{a_1}}}\n{g}{L^{p_2}_{\langle x \rangle^{a_2}}} +   \n{ \nabla  f}{L^{p_1}_{\langle x \rangle^{a_1}}}\n{J^{s-1} g}{L^{p_2}_{\langle x \rangle^{a_2}}}.
\end{align*}

While the above argument for $II_{2.1}$ works only for $s\in (0,1)$, the estimate for $II_{2.2}$ given below will work for both $s\in (0,1)$ and $s>1$.  Applying H\"older's inequality, we get
\begin{align*}
\n{\De_k f\, J^s    S_{k-3}g}{L^p_{\langle x \rangle^a}} &\lesssim  \n{\De_k f}{L^{p_1}_{\langle x \rangle^{a_1}}}\n{J^{s}S_{k-3} g}{L^{p_2}_{\langle x \rangle^{a_2}}}.
\end{align*}
By Lemma~\ref{Bern}(iii) we can write
\begin{equation*}
\n{J^{s}S_{k-3} g}{L^{p_2}_{\langle x \rangle^{a_2}}} \lesssim \min \left\{  2^{k}  \n{J^{s-1} g}{L^{p_2}_{\langle x \rangle^{a_2}}},2^{ks}  \n{ g}{L^{p_2}_{\langle x \rangle^{a_2}}} \right\},
\end{equation*}
which, together with \eqref{eq:equiv}, yields
\begin{equation*}
\n{\De_k f\, J^s    S_{k-3}g}{L^p_{\langle x \rangle^a}}  \lesssim \min\left\{  2^{k(1-s)}  \n{J^{s-1} g}{L^{p_2}_{\langle x \rangle^{a_2}}} \n{J^s f}{L^{p_1}_{\langle x \rangle^{a_1}}},  2^{-k(1-s)}  \n{ g}{L^{p_2}_{\langle x \rangle^{a_2}}} \n{ \na f}{L^{p_1}_{\langle x \rangle^{a_1}}}\right\}.
\end{equation*}
Then, we can close this argument for either $0<s<1$ or $s>1$ by applying Lemma \ref{le:2.6A}.  This completes the proof of \eqref{eq:new1.4}.

\subsection{Proof of \eqref{eq:new1.5}}\label{sec:com3}

In this subsection, we will prove \eqref{eq:new1.5} under the assumption that $s>1$.  The LHS of this inequality contains an extra term, which we will consider in the estimate for $II_1$.  We denote
\[
\wt{II_1} := \sum_{k\geq 0} [J^s, S_{k-3} f] \De_k g  - s\na f\cdot \na J^{s-2} \De_{>0} g,
\]
so that the decomposition for the LHS of \eqref{eq:new1.5} will replace \eqref{eq:decomp} with:
\[
\wt{II_1} + II_2 + II_3 + II_4  - s\na f\cdot \na J^{s-2} S_0 g,
\]
where $II_2,II_3,$ and $II_4$ are the same as in Subsection \ref{sec:com2}.
Since the Fourier symbol corresponding to $\na J^{-1} S_0$ is Schwartz, we can easily estimate the last term above using the H\"older and weighted Young inequalities.  We also note that the estimates for $II_3$ and $II_4$ from Subsection~\ref{sec:com2} still hold valid when $s>1$, so we will omit these computations.  We treat now the remaining terms $\wt{II_1}$ and $II_2$.
\\

\noindent
\textbf{Estimate for $II_2$:}\\
 The estimate for $II_2$ in this case is not much different from the previous case $0<s<1.$
We write
\begin{align*}
[J^s , \De_k f]  S_{k-3} g &= J^s  (\De_k f  S_{k-3} g)- \De_k f\, J^s    S_{k-3}g\\
 &= [J^s, S_{k-3}g]  \De_k f  + J^s \De_k f  S_{k-3} g - \De_k f\, J^s    S_{k-3}g,
\end{align*}
which leads to
\[
II_2 = \sum_{k\geq 0} [J^s, S_{k-3} g] \De_k f + g J^s \De_{>0} f - \sum_{k\geq 0}J^s \De_{k} f \De_{>k-3} g - \sum_{k\geq 0}\De_k f\, J^s    S_{k-3}g =: II_{2.3} + g J^s \De_{>0} f - II_{2.4} - II_{2.5}.
\]
Indexing above begins with $2.3$ rather than $2.1$ to prevent confusion with the notations introduced in Subsection~\ref{sec:com2}.  Here, $g J^s \De_{>0} f$ can be estimated by the H\"{o}lder and weighted Young inequalities, while the estimate for $II_{2.4}$ is analogous to that of $II_{3}$ where we can apply statement (i) from Lemma~\ref{Bern} rather than (ii).  Also, $II_{2.5}$ above is identical to $II_{2.2}$ from Subsection~\ref{sec:com2}, whose estimate is valid for $s>1$ as mentioned previously.  Hence, the details for these terms are omitted.

It remains to treat $II_{2.3}$.  Applying Proposition \ref{prop:key2}, we have
\begin{equation*}
\n{[J^s, S_{k-3}g] \De_k f}{L^p_{\langle x \rangle^a}} \lesssim   2^{k(s-1)}  \n{\De_k f}{L^{p_1}_{\langle x \rangle^{a_1}}}  \n{\na S_{k-3} g}{L^{p_2}_{\langle x \rangle^{a_2}}}.
\end{equation*}
We know that
\[
\n{\na S_{k-3} g}{L^{p_2}_{\langle x \rangle^{a_2}}} \lesssim 2^{k} \n{g}{L^{p_2}_{\langle x \rangle^{a_2}}}.
\]
Also, applying Lemma~\ref{Bern}(iv) with $s$ replaced by $s-1$, we see that for $0<\ve<\min\{1,s-1\}$
\[
\n{\na S_{k-3} g}{L^{p_2}_{\langle x \rangle^{a_2}}} \lesssim 2^{k(1-\ve)} \n{g}{L^{p_2}_{\langle x \rangle^{a_2}}}^{\f{s-1-\ve}{s-1}} \n{{J^{s-1} g}}{L^{p_2}_{\langle x \rangle^{a_2}}}^{\f{\ve}{s-1}}.
\]
These estimates, together with \eqref{eq:equiv}, yield
\[
\n{[J^s, S_{k-3}g] \De_k f}{L^p_{\langle x \rangle^a}} \lesssim \min \left\{ 2^{k(s-1)}\n{\na f}{L^{p_1}_{\langle x \rangle^{a_1}}} \n{g}{L^{p_2}_{\langle x \rangle^{a_2}}},  2^{-k\ve}\n{J^s f}{L^{p_1}_{\langle x \rangle^{a_1}}}\n{g}{L^{p_2}_{\langle x \rangle^{a_2}}}^{\f{s-1-\ve}{s-1}} \n{J^{s-1}g}{L^{p_2}_{\langle x \rangle^{a_2}}}^{\f{\ve}{s-1}}\right\}.
\]
Applying Lemma \ref{le:2.6A} with $a= s-1$ and $b=\ve$ as well as
\[
 A =\n{\na f}{L^{p_1}_{\langle x \rangle^{a_1}}} \n{g}{L^{p_2}_{\langle x \rangle^{a_2}}}, \quad B = \n{J^s f}{L^{p_1}_{\langle x \rangle^{a_1}}}\n{g}{L^{p_2}_{\langle x \rangle^{a_2}}}^{\f{s-1-\ve}{s-1}} \n{J^{s-1}g}{L^{p_2}_{\langle x \rangle^{a_2}}}^{\f{\ve}{s-1}}
\]
leads to
\begin{align*}
\|II_{2.3}\|_{L^p_{\langle x \rangle^a}} &\lesssim \pr{\n{\na f}{L^{p_1}_{\langle x \rangle^{a_1}}} \n{g}{L^{p_2}_{\langle x \rangle^{a_2}}} }^{\frac{\ve}{s-1 +\ve}}  \pr{ \n{J^s  f}{L^{p_1}_{\langle x \rangle^{a_1}}}  \n{g}{L^{p_2}_{\langle x \rangle^{a_2}}}^{\frac{s-1-\ve}{s-1}}\n{J^{s-1}g}{L^{p_2}_{\langle x \rangle^{a_2}}}^{\frac{\ve}{s-1}}}^{\f{s-1}{s-1+\ve}}\\
&= \pr{\|J^s f\|_{L^{p_1}_{ \langle x\rangle^{a_1}}} \|g\|_{L^{p_2}_{\langle x \rangle^{a_2}}}}^{\frac{s-1}{s-1+\ve}} \pr{\|\na f\|_{L^{p_1}_{\langle x \rangle^{a_1}}}
\|J^{s-1}g\|_{L^{p_2}_{\langle x \rangle^{a_2}}}}^{\frac{\ve}{s-1+\ve}} \\
&  \lesssim \|J^s f\|_{L^{p_1}_{\langle x \rangle^{a_1}}} \|g\|_{L^{p_2}_{\langle x \rangle^{a_2}}}+ \|\na f\|_{L^{p_1}_{\langle x \rangle^{a_1}}} \|J^{s-1} g\|_{L^{p_2}_{\langle x \rangle^{a_2}}}.
\end{align*}
Altogether, we obtain the desired estimate for $II_2$ when $s>1$.\\

\noindent
\textbf{Estimate for $\wt{II_1}$:}\\
This estimate is the most involved and requires a more complicated commutator structure, involving  $s\na f\cdot \na J^{s-2} \De_{>0} g$, to ensure summability in $k\in \nn$.   We write $\wt{II_1}$ as
\[
\sum_{k\geq 0} \pr{ [J^s, S_{k-3} f] \De_k g- s\na S_{k-3} f \cdot \na J^{s-2} \De_k g}  + s\sum_{k\geq 0} \na\De_{>k-3} f \cdot \, \na J^{s-2} \De_k g =: \wt{II_{1.1}} + \wt{II_{1.2}}.
\]
$\wt{II_{1.2}}$ can be handled similar to $II_3$ in Subsection~\ref{sec:com2}. Indeed, for $s>1,$ Lemmas~\ref{lp}, \ref{Bern}(i), (ii), \ref{lem:4.2}(i) and (ii) lead to
\[
\n{\na\De_{>k-3} f \cdot \, \na J^{s-2} \De_k g}{L^{p}_{\langle x\rangle^a}}\lesssim \min \left\{2^{k(s-1)} \n{\na f}{L^{p_1}_{\langle x\rangle^{a_1}}}\n{ g}{L^{p_2}_{\langle x\rangle^{a_2}}},  2^{-k(s-1)} \n{J^s f}{L^{p_1}_{\langle x\rangle^{a_1}}} \n{J^{s-1} g}{L^{p_2}_{\langle x\rangle^{a_2}}}\right\},
\]
where the RHS is the same as that in \eqref{eq:ii3}.  Hence, the desired inequality follows using the same computations.

It remains to estimate  $\wt{II_{1.1}}$, which requires the second-order extension of Proposition~\ref{prop:key2}.  We introduce the following proposition.

\begin{prop}\label{prop:key3}
Let $a_1,a_2\ge 0$, $p\in [\f{1}{2}, 1]$ and $p_1,p_2\in [1,\infty]$ satisfying $ \f{1}{p_1} + \f{1}{p_2}=\f{1}{p}$ and $\f{a_1}{p_1} + \f{a_2}{p_2} = \f{a}{p}$.  Then for any $s > 0$ and $k\geq 0$,
\begin{equation}\label{eq:key3}
\n{[J^s, S_{k-3}f] \De_k g- s \na S_{k-3} f \cdot \na J^{s-2} \De_k g}{L^p_{\langle x \rangle^a}} \lesssim   2^{k(s-2)}  \n{|\na|^2 S_{k-3} f}{L^{p_1}_{\langle x \rangle^{a_1}}} \n{\De_k g}{L^{p_2}_{\langle x \rangle^{a_2}}},
\end{equation}
where $|\na|^2=-\Delta$.
\end{prop}

\begin{proof}
We can write $[J^s, S_{k-3}f] \De_k g- s \na S_{k-3} f \cdot \na  J^{s-2} \De_k g$ as follows:
\begin{align*}
&\int_{\rr^{n+n}} \pr{\langle\xi+\eta\rangle^s - \langle\eta\rangle^s - s\xi\cdot \eta \lan{\eta}^{s-2}}  \wh{S_{k-3} f}(\xi)  \wh{\De_k g}(\eta) e^{i(\xi+\eta)\cdot x}\, d\xi \, d\eta\\
&=\int_{\rr^{n+n}} \pr{\int_0^1s \xi \cdot (t\xi+\eta) \langle t\xi + \eta\rangle^{s-2} - s\xi\cdot \eta \lan{\eta}^{s-2} \,dt}  \wh{S_{k-3} f}(\xi) \wh{\De_k g}(\eta) e^{i(\xi+\eta)\cdot x}\, d\xi \, d\eta\\
&=\int_{\rr^{n+n}} \pr{\int_0^1\int_0^t s |\xi|^2  \pr{\langle t'\xi + \eta\rangle^{s-2}+  |t'\xi+\eta|^2 \langle t'\xi + \eta\rangle^{s-4}}\,dt' \,dt}  \wh{S_{k-3} f}(\xi) \wh{\De_k g}(\eta) e^{i(\xi+\eta)\cdot x}\, d\xi \, d\eta\\
&=s\int_0^1\int_0^t \int_{\rr^{n+n}}  \pr{\langle t'\xi + \eta\rangle^{s-2}+  |t'\xi+\eta|^2 \langle t'\xi + \eta\rangle^{s-4}} \,\wh{|\na|^2 S_{k-3} f}(\xi) \wh{\De_k g}(\eta) e^{i(\xi+\eta)\cdot x}\, d\xi \, d\eta\,dt' \,dt.
\end{align*}
We can follow the computations within the proof of Proposition~\ref{prop:key2}, including the application of Lemma~\ref{le:commutator}, to obtain the desired estimate.  Doing so involves a trivial modification of Lemma~\ref{le:commutator} to allow for the double integral $dt'\, dt$.  We omit the details.
\end{proof}

Let us continue with the estimate of $\wt{II_{1.1}}$.  By Lemma \ref{Young}, we can write
\[
\n{|\na|^2 S_{k-3} f }{L^{p_1}_{\langle x \rangle^{a_1}}} = \n{2^{(k-3)n} \pr{\na \wh{\Phi}(2^{k-3}\cdot)} * (\na f) }{L^{p_1}_{\langle x \rangle^{a_1}}} \lesssim 2^{k} \n{\na f }{L^{p_1}_{\langle x \rangle^{a_1}}},
\]
where $\na g * \na f=\sum_{i=1}^n \partial_i g * \partial_i f.$
Combining this with Lemma~\ref{lem:4.2}(iii), for some $0<\ve< \min(1,s-1)$, we have
\[
\n{|\na|^2 S_{k-3} f }{L^{p_1}_{\langle x \rangle^{a_1}}} \lesssim \min \left\{2^{k} \n{\na f }{L^{p_1}_{\langle x \rangle^{a_1}}}, 2^{k(1-\ve)} \n{\na f }{L^{p_1}_{\langle x \rangle^{a_1}}}^{\f{s-1-\ve}{s-1}}\n{J^{s} f }{L^{p_1}_{\langle x \rangle^{a_1}}}^{\f{\ve}{s-1}}\right\}.
\]
Substituting this bound into \eqref{eq:key3} and applying Lemmas~\ref{lp} and \ref{Bern}(ii) to $g$, we have
\begin{align*}
\n{[J^s, S_{k-3}f] \De_k g- s \na S_{k-3} f \cdot \na J^{s-2} \De_k g}{L^p_{\langle x \rangle^a}}&  \\
&\hspace{-100pt}\lesssim  \min \left\{2^{k(s-1)}\n{\na f}{L^{p_1}_{\langle x \rangle^{a_1}}} \n{g}{L^{p_2}_{\langle x \rangle^{a_2}}},\, 2^{-k\ve} \n{\na f }{L^{p_1}_{\langle x \rangle^{a_1}}}^{\f{s-1-\ve}{s-1}}\n{J^{s} f }{L^{p_1}_{\langle x \rangle^{a_1}}}^{\f{\ve}{s-1}} \n{J^{s-1} g}{L^{p_2}_{\langle x \rangle^{a_2}}}  \right\}.
\end{align*}
Finally, we apply Lemma~\ref{le:2.6A} with $a=s-1$ and $b=\ve$, thereby obtaining
\begin{align*}
\n{\wt{I_{1.1}}}{L^p_{\langle x \rangle^a}} &\lesssim \pr{\n{\na f}{L^{p_1}_{\langle x \rangle^{a_1}}} \n{ g}{L^{p_2}_{\langle x \rangle^{a_2}}} }^{\frac{\ve}{s-1+\ve}}  \pr{ {\n{ J^{s-1} g}{L^{p_2}_{\langle x \rangle^{a_2}}}}  \n{\na f}{L^{p_1}_{\langle x \rangle^{a_1}}}^{\frac{s-1-\ve}{s-1}}\n{J^s f}{L^{p_1}_{\langle x \rangle^{a_1}}}^{\frac{\ve}{s-1}}}^{\f{s-1}{s-1+\ve}}\\
&= \pr{{\n{ J^{s-1} g}{L^{p_2}_{\langle x \rangle^{a_2}}}} \|\na f\|_{L^{p_1}_{\langle x \rangle^{a_1}}}}^{\frac{s-1}{s-1+\ve}} \pr{\|J^s f\|_{L^{p_1}_{\langle x \rangle^{a_1}}}
\|g\|_{L^{p_2}_{\langle x \rangle^{a_2}}}}^{\frac{\ve}{s-1+\ve}} \\
&  \lesssim  \|J^s f\|_{L^{p_1}_{\langle x \rangle^{a_1}}} \|g\|_{L^{p_2}_{\langle x \rangle^{a_2}}}+ \|\na f\|_{L^{p_1}_{\langle x \rangle^{a_1}}} {\n{ J^{s-1} g}{L^{p_2}_{\langle x \rangle^{a_2}}}},
\end{align*}
This completes the proof of
$\wt{II_{1.1}}$, and hence \eqref{eq:new1.5} in Theorem~\ref{main-comm}.

\subsection{Proof of the Sharpness of the Lower Threshold for $s$}\label{sec:com4}

In this subsection, we will prove the negative direction of Theorem~\ref{main-comm}.  We note that applying (quasi-)triangle and H\"older inequalities to \eqref{eq:new1.4} and \eqref{eq:new1.5} leads repectively to
\begin{align}
\n{J^s (fg) }{L^p_{ \langle x\rangle^a
}} &\lesssim \tn{RHS of \eqref{1.4}} +  \n{\na f}{L^{p_1}_{\langle  x \rangle^{a_1}}}\n{J^{s-1} g}{L^{p_2}_{\langle  x \rangle^{a_2}}},\label{eq:com1} \\
\n{J^s (fg) }{L^p_{ \langle x\rangle^a
}} &\lesssim \tn{RHS of \eqref{1.4}} +  \n{\na f}{L^{p_1}_{\langle  x \rangle^{a_1}}}\pr{ \n{J^{s-1} g}{L^{p_2}_{\langle  x \rangle^{a_2}}} + \n{ \na J^{s-2} g}{L^{p_2}_{\langle x \rangle^{a_2}}}  }.\label{eq:com2}
\end{align}

For $s< 0$, we can again take
\[
f(x)=f_k(x)=e^{i2^k e_1\cdot x} \wh{\Phi}(x),\qquad g(x)=g_k(x)=e^{-i2^k e_1\cdot x} \wh{\Phi}(x),
\]
as in Subsection~\ref{sec:inhom3}. Since $k\in \nn$, a direct computation gives $\n{\na f}{L^{p_1}_{\langle  x \rangle^{a_1}}} \lesssim 2^k$, while the same arguments from  Subsection~\ref{sec:inhom3} gives $\n{J^{s-1} g}{L^{p_2}_{\langle  x \rangle^{a_2}}} \lesssim 2^{(s-1)k}$ and $\n{\na J^{s-2} g}{L^{p_2}_{\langle  x \rangle^{a_2}}} \lesssim 2^{(s-1)k}.$  Hence, the RHS of \eqref{eq:com1} and \eqref{eq:com2} approach zero as $k\to \infty$ when $s<0$, while the LHS remains a positive constant, leading to a contradiction.

Next, we assume $0 < s \leq n\pr{\f{1}{p} - 1}$ and $s \not\in 2\nn$.  For the commutator estimates, the counterexample in Subsection~\ref{sec:inhom3} cannot be used due to the fact that the RHS of \eqref{eq:com1} and \eqref{eq:com2} may contain negative differential operators $J^{s-1}$ or $\na J^{s-2}$.  Dilated forms of  these operators, $J_{\de}^{s-1}$ or $\de \na J^{s-2}_{\de}$ may be unbounded on $L^{p_2}_{\lan{x}^{a_2}}$, which invalidates the previous counterexample.

However, we can still generate a counterexample similarly and avoid this issue.  Rather than choosing arbitrary functions $f,g \in C^\infty_c(\rr^n)$, we choose $f = g = \Psi_\de = \Psi(\cdot/\de)$, where $\Psi$ is the Schwartz function introduced in Section~\ref{sec:prelim} whose Fourier transform is supported on an annulus.  Applying either \eqref{eq:com1} and \eqref{eq:com2} to this choice of $f$ and $g$ leads respectively to
\begin{align*}
\n{J^s_\de (\Psi)^2 }{L^p_{ \langle \de x\rangle^a
}} &\lesssim \n{J^s_\de \Psi}{L^{p_1}_{\langle \de x \rangle^{a_1}}} \n{\Psi}{L^{p_2}_{\langle \de x \rangle^{a_2}}} + \n{ \Psi}{L^{p_1}_{\langle \de x \rangle^{a_1}}}\n{J^{s}_\de \Psi}{L^{p_2}_{\langle\de  x \rangle^{a_2}}} + \n{\na \Psi}{L^{p_1}_{\langle \de x \rangle^{a_1}}}\n{J^{s-1}_\de \Psi}{L^{p_2}_{\langle \de x \rangle^{a_2}}}, \\
\n{J^s_\de (\Psi)^2 }{L^p_{ \langle\de  x\rangle^a
}} &\lesssim   \n{J^s_\de \Psi}{L^{p_1}_{\langle \de x \rangle^{a_1}}} \n{\Psi}{L^{p_2}_{\langle \de x \rangle^{a_2}}} + \n{ \Psi}{L^{p_1}_{\langle \de x \rangle^{a_1}}}\n{J^{s}_\de \Psi}{L^{p_2}_{\langle\de  x \rangle^{a_2}}} + \n{\na \Psi}{L^{p_1}_{\langle \de  x \rangle^{a_1}}}\pr{ \n{J^{s-1}_\de \Psi}{L^{p_2}_{\langle \de x \rangle^{a_2}}} +\n{ \na J_\de^{s-2} \Psi}{L^{p_2}_{\langle \de x \rangle^{a_2}}}  }.
\end{align*}

Since $\Psi$ is real-valued, $(\Psi)^2$ is non-negative and non-zero Schwartz function.  
On the RHS of either inequality, we see some $J^s_\de$ (or $J^{s-1}_\de$, $J^{s-2}_\de$)  applied to  $\Psi$.  Similar to the discussion in the proof of Proposition~\ref{prop:key2}, we can show that, for any $\al \in \rr$,
\[
|J^\al_{\de} \Psi (x)| \lesssim_{M,\al} \lan{x}^{-M},
\]
where the implicit constant is independent of $\de \in (0,1]$.  This shows that the RHS of either inequality is uniformly bounded for $\de \in (0,1]$, while we have shown in  Subsection~\ref{sec:inhom3} that the LHS of either inequality becomes unbounded as $\de \to 0$ if $0 < s \leq n\pr{\f{1}{p} - 1}$ and $s \not\in 2\nn$.  This prove the sharpness of the restriction $s> \max\left\{n\pr{\f{1}{p} - 1},0\right\}$ and $s\not\in 2\nn$.

\section{Biparameter Kato-Ponce Inequality in the  Weighted Mixed Norm Setting}\label{sec:multi}
We note that all of the lemmas in Section~\ref{sec:inhom} extend effortlessly to the mixed norm setting when the weighted Lebesgue spaces $L^p_{\lan{x}^a}$, $L^{p_1}_{\lan{x}^{a_1}}$, $L^{p_2}_{\lan{x}^{a_2}}$ are replaced by the weighted mixed Lebesgue spaces $\lp$, $\lpa$, $\lpb$, respectively.  Hence, Theorem~\ref{main2} is proved in the same way as Theorem~\ref{main1}, and we omit the details.

On the other hand, the proof of Theorem~\ref{main3} requires additional work. The proof of Theorem~\ref{main3} in the unweighted setting has been presented in \cite{OW}.  In this section, we will provide necessary modifications to the arguments given in \cite{OW} to extend the theorem to the weighted setting.  We introduce the notations to be used henceforth:
\begin{itemize}
\item For spatial variables, we will use $\dot{x}\in \rr^{\dot{n}}$, $\ddot{x}\in \rr^{\ddot{n}}$, and $\wt{x} = (\dot{x},\dot{x}')\in \rr^{\wt{n}}$ where $\wt{n}= \dot{n}+ \ddot{n}$.
\item For frequency variables, we will use  $\xi \in \rr^{\dot{n}}$, $\xi'\in \rr^{\ddot{n}}$, and $\wt{\xi} = (\xi, \ddot{\xi})\in \rr^{\wt{n}}$.
\item Prime and tilde accompanying parameter values represent the dimensionality stated above.  For instance, $L^{\dot{p}}_{\dot{a}} = L^p_{\lan{x}^a}(\rr^n)$ and $L^{p'}_{a'}= L^p_{\lan{x'}^{a'}}(\rr^{n'})$.
\item Prime and tilde accompanying operators represent the dimensionality stated above.   For instance $S_k$ is an operator for $x\in \rr^n$, while ${S}'_{{k}'}$ and $\wt{S}_{\wt{k}}$ are the analogous operators for ${{x}'} \in \rr^{{{n}'}}$ and $\wt{x} \in \rr^{\wt{n}}$, respectively.
\item When an operator already has an exponent, only the exponent will accompany the prime or tilde.  For instance, $J^s$, $J^{s'}$ and $J^{\wt{s}}$ represents the corresponding operators for $\rr^n$, $\rr^{n'}$, and $\rr^{\wt{n}}$, respectively.
\item Finally, given $p, p' \in (0,\infty]$, we use the following notations:
\[
\bar{p} := \min\{1, p\}, \qquad p^* = \min \{1,p,p'\}.
\]
\end{itemize}

We will take for granted the extensions of Lemmas~\ref{Young}, \ref{lp} and \ref{Bern} in Subsection~\ref{sec:inhom1} into the weighted mixed norm setting.  For instance, the weighted Young's inequality given in Lemma~\ref{Young} extends easily to the mixed-norm setting:
\[
\|f*g\|_{L^{\ddot{r}}_{\ddot{a}}L^{\dot{r}}_{\dot{a}}}\le \|f\|_{L^{\ddot{p}}_{\ddot{a}}L^{\dot{p}}_{\dot{a}}} \|g\|_{L^{\ddot{q}}_{\ddot{a}}L^{\dot{q}}_{\dot{a}}},
\]
where the parameters satisfy appropriate conditions.  Also, the subadditivity of the mixed norm given in \cite{OW} also easily extends to the weighted setting.  More specifically, for any $\dot{p},\ddot{p}>0$ and $\dot{a},\ddot{a} \in \rr$, we have
\begin{equation}\label{eq:mixedtri}
\Big\|\sum_k f_k \Big\|_{\lp}^{p^*} \leq \sum_k \n{f_k}{\lp}^{p^*}.
\end{equation}
A less trivial extension is that of our main lemma, Lemma~\ref{le:main}, whose modification is given below.

\begin{lem}\label{mixed-main}
Let $\dot{p}, \ddot{p}\in (0,\infty]$, $a,a'\in [0,\infty)$, $\{\dot{\si}_{\dot{k}}\}_{\dot{k}\in\nn}$ and $\{\ddot{\si}_{\ddot{k}}\}_{\ddot{k}\in\nn}$ be two compactly supported families of functions on $\rr^{\dot{n}}$ and $\rr^{\ddot{n}}$, respectively (in the sense of \eqref{new3.11}).  If there exist constants $\dot{c}, c'>0$ and $\dot{\ga}, \ddot{\ga}$ satisfying $\ds \dot{\ga} >  \f{\dot{n}}{\bar{p}}$ and $\ds \ddot{\ga} > \f{\ddot{n}}{p^*}$ such that
\begin{align}\label{new5.2}
\begin{split}
 \abs{\wh{\dot{\si}_{\dot{k}}}\pr{\dot{x}} } &\lesssim_{\dot{M}} \lan{\dot{x}}^{-\dot{M}} +  \lan{\dot{x}}^{-\dot{\ga}} e^{-\dot{c} 2^{-\dot{k}}\abs{\dot{x}}}, \\
 \abs{\wh{\ddot{\si}_{\ddot{k}}}\pr{\ddot{x}} } &\lesssim_{\ddot{M}} \lan{\ddot{x}}^{-\ddot{M}} +  \lan{\ddot{x}}^{-\ddot{\ga}} e^{-c' 2^{-\ddot{k}}\abs{\ddot{x}}},
 \end{split}
\end{align}
for any $\dot{M}, \ddot{M}\gg 1$ and $\dot{k}, \ddot{k}\in \bn_+$, where the implicit constants are independent of $\dot{k}, \ddot{k}$, then
\begin{align}
\n{ \int_{\rr^{{\wt n}}} \dot{\si}_{\dot{k}} \pr{2^{-\dot{k}} {\dot{\xi}}}\ddot{\si}_{\ddot{k}} \pr{2^{-\ddot{k}} {\ddot{\xi}}} \cf \left[\ddot{S}_{\ddot{k}} \dot{S}_{\dot{k}} h\right]\pr{{\wt\xi}}  e^{i{\wt\xi} \cdot {\wt x}}\, d{\wt \xi}}{\lp} &\lesssim \n{\ddot{S}_{\ddot{k}} \dot{S}_{\dot{k}} h}{\lp},\label{eq:bip1}\\
\n{ \int_{\rr^{\dot{n}}} \dot{\si}_{\dot{k}} \pr{2^{-\dot{k}} {\dot{\xi}}} \dot{\cf}\left[\dot{S}_{\dot{k}} h\right]\pr{\dot{\xi},\ddot{x}}  e^{i\dot{\xi} \cdot \dot{x}}\, d\dot{\xi}}{\lp} &\lesssim \n{\dot{S}_{\dot{k}} h}{\lp},\label{eq:bip2}\\
\n{ \int_{\rr^{\ddot{n}}} \ddot{\si}_{\ddot{k}} \pr{2^{-\ddot{k}} {\ddot{\xi}}} \ddot{\cf}\left[\ddot{S}_{\ddot{k}} h\right]\pr{\dot{x},\ddot{\xi}}  e^{i\ddot{\xi} \cdot \ddot{x}}\, d\ddot{\xi}}{\lp} &\lesssim \n{\ddot{S}_{\ddot{k}} h}{\lp}. \label{eq:bip3}
\end{align}
\end{lem}

\begin{proof}
Both \eqref{eq:bip2} and \eqref{eq:bip3} can be shown using the same arguments as for Lemma~\ref{le:main},  by simply replacing the regular norm with the mixed norm.  We will only highlight necessary additions to prove the true biparameter variant \eqref{eq:bip1}. We assume
$$\supp~\si_k\subset [-R,R]^n, \quad \supp~\si_{k'}'\subset [-R',R']^{n'},\quad \forall k,k'\in \bn,$$
and expanding the symbols $\si_{k}$ and $\si_{k}'$ into the Fourier series on these cubes. Then decays of their Fourier coefficients $\dot{c}_{\dot{\mm}}^{\dot{k}}$ and $c_{\ddot{\mm}}^{\ddot{k}}$ are given by
\begin{align*}
 \abs{\dot{c}_{\dot{\mm}}^{\dot{k}} } &\lesssim_{\dot{M}} \lan{\dot{\mm}}^{-\dot{M}} +  \lan{\dot{\mm}}^{-\dot{\ga}} e^{-\dot{c} 2^{-\dot{k}}\abs{\dot{\mm}}}, \\
 \abs{c_{\ddot{\mm}}^{\ddot{k}}} &\lesssim_{\ddot{M}} \lan{\ddot{\mm}}^{-\ddot{M}} +  \lan{\ddot{\mm}}^{-\ddot{\ga}} e^{-c 2^{-\ddot{k}}\abs{\ddot{\mm}}},
\end{align*}
with $\ds \dot{\ga} >  \f{\dot{n}}{\bar{p}}$ and $\ds \ddot{\ga} > \f{\ddot{n}}{p^*}$, due to our hypothesis \eqref{new5.2}.  Following the computations in the proof of Lemma~\ref{le:main}, we obtain
\[
 \int_{\rr^{\wt{n}}} \dot{\si}_{\dot{k}} \pr{2^{-\dot{k}} {\dot{\xi}}}\ddot{\si}_{\ddot{k}} \pr{2^{-\ddot{k}} {\ddot{\xi}}} \cf \left[\ddot{S}_{\ddot{k}} \dot{S}_{\dot{k}} h\right]\pr{\wt{\xi}}  e^{i\wt{\xi} \cdot \wt{x}}\, d\wt{\xi} = \sum_{\tiny\begin{array}{c} \dot{\mm}\in \zz^{\dot{n}}\\ \ddot{\mm}\in \zz^{\ddot{n}}\end{array}} \dot{c}_{\dot{\mm}}^{\dot{k}} c_{\ddot{\mm}}^{\ddot{k}} \ddot{S}_{\ddot{k}} \dot{S}_{\dot{k}} h\pr{\dot{x} - \f{2\pi}{\dot{R}} 2^{-\dot{k}}\dot{\mm}, \ddot{x} - \f{2\pi}{\ddot{R}} 2^{-\ddot{k}}\ddot{\mm} }.
\]

We first take the $L^{\dot{p}}_{\lan{\dot{x}}^{\dot{a}}}$ norm of the RHS above, and then raise it to the power $\bar{p} = \min \pr{1,\dot{p}}$.  The resulting expression is bounded from above by
\[
C_{\dot{R}} \sum_{\dot{\mm}\in\zz^{\dot{n}}} \abs{ \dot{c}_{\dot{\mm}}^{\dot{k}}}^{\bar{p}}  \lan{2^{-\dot{k}}\dot{\mm}}^{\f{\dot{a}\bar{p}}{\dot{p}}} \n{\sum_{ \ddot{\mm}\in \zz^{\ddot{n}}}  c_{\ddot{\mm}}^{\ddot{k}} \ddot{S}_{\ddot{k}} \dot{S}_{\dot{k}} h\pr{\dot{x}, \ddot{x} - \f{2\pi}{\ddot{R}} 2^{-\ddot{k}}\ddot{\mm}}}{L^{\dot{p}}_{\lan{\dot{x}}^{\dot{a}}}}^{\bar{p}}.
\]
As before, the summation in $\dot{\mm}$ converges uniformly in $k\in \nn$ as long as $\ds \dot{\ga} > \f{\dot{n}}{\bar{p}}$.  This leads to
\[
\n{ \int_{\rr^{{\wt n}}} \dot{\si}_{\dot{k}} \pr{2^{-\dot{k}} {\dot{\xi}}}\ddot{\si}_{\ddot{k}} \pr{2^{-\ddot{k}} {\ddot{\xi}}} \cf \left[\ddot{S}_{\ddot{k}} \dot{S}_{\dot{k}} h\right]\pr{{\wt\xi}}  e^{i{\wt\xi} \cdot {\wt x}}\, d{\wt \xi}}{L^p_{\lan{x}^a}}  \lesssim  \n{\sum_{ \ddot{\mm}\in \zz^{\ddot{n}}}  c_{\ddot{\mm}}^{\ddot{k}} \ddot{S}_{\ddot{k}} \dot{S}_{\dot{k}} h\pr{\dot{x}, \ddot{x} - \f{2\pi}{\ddot{R}} 2^{-\ddot{k}}\ddot{\mm}}}{L^{\dot{p}}_{\lan{\dot{x}}^{\dot{a}}}}.
\]
Taking the $L^{\ddot{p}}_{\lan{\ddot{x}}^{\ddot{a}}}(\rr^{\ddot{n}})$ norm first, and then raising it to the power $p^*$, the resulting expression is bounded by a constant multiple of
\[
\sum_{ \ddot{\mm}\in \zz^{\ddot{n}}}  \abs{c_{\ddot{\mm}}^{\ddot{k}}}^{p^*}    \lan{2^{-\ddot{k}}\ddot{\mm}}^{\f{\ddot{a}p^*}{\ddot{p}}} \n{ \ddot{S}_{\ddot{k}} \dot{S}_{\dot{k}} h }{\lp}^{p^*}.
\]
The summation in $\ddot{\mm}$ converges given $\ds \ddot{\ga} > \f{\ddot{n}}{p^*}$, which leads to the desired inequality \eqref{eq:bip1}.
\end{proof}

Both Lemma~\ref{le:commutator} and Proposition~\ref{prop:key2} extend effortlessly to the mixed norm setting, so we omit the details.  In particular, the following commutator estimates hold:
\begin{align*}
\n{[J^{\dot{s}},    \dot{S}_{\dot{k} -3}   f]  \dot{\De}_{\dot{k}} g}{\lp} &\lesssim 2^{\dot{k}\pr{\dot{s}-1}}  \n{\dot{\na} \dot{S}_{\dot{k}-3} f}{\lpa} \n{\dot{\De}_{\dot{k}} g}{\lpb},  \\
\n{[J^{\ddot{s}},    \ddot{S}_{\ddot{k} -3}   f]  \ddot{\De}_{\ddot{k}} g}{\lp} &\lesssim 2^{\ddot{k}\pr{\ddot{s}-1}} \n{\ddot{\na} \ddot{S}_{\ddot{k}-3} f}{\lpa} \n{\ddot{\De}_{\ddot{k}} g}{\lpb}.
\end{align*}

However, as seen in \cite{OW}, biparameter Kato-Ponce inequality involves more complicated types of commutator estimates.  Below, we will state and prove the corresponding estimates in the weighted setting, which will be the final piece of nontrivial modifications required to establish Theorem~\ref{main3}.

\begin{prop} \label{prop:multi2}
Let all indices satisfy the same conditions as in Lemma \ref{mixed-main}. For each $k,k'\in\nn$, let $T_1^{k,k'}$ and $T_2^{k,k'}$ be two families of bilinear operators defined as follows:
\begin{align*}
T_1^{k,k'}(f,g)\pr{\wt{x}} &:=\int_{\rr^{2\wt{n}}} \pr{\lan{\dot{\xi}+\dot{\eta}}^{\dot{s}}- \lan{\dot{\eta}}^{\dot{s}}} \left( \lan{\ddot{\xi}+\ddot{\eta}}^{\ddot{s}}-\lan{\ddot{\eta}}^{\ddot{s}} \right) \\
&\qquad \qquad \times \wt{\cf} \left[\dot{S}_{\dot{k}-3} \ddot{S}_{\ddot{k}-3} f \right]\pr{\dot{\xi},\ddot{\xi}} \cf\left[ \dot{\De}_{\dot{k}}\ddot{\De}_{\ddot{k}} g \right]\pr{\dot{\eta},\ddot{\eta}} e^{i\wt{x} \cdot (\wt{\xi}+\wt{\eta})} \, d \wt{\xi} \,d\wt{\eta}, \\
T_2^{k,k'}(f,g)\pr{\wt{x}} &:=\int_{\rr^{2\wt{n}}} \pr{\lan{\dot{\xi}+\dot{\eta}}^{\dot{s}}- \lan{\dot{\xi}}^{\dot{s}}} \left( \lan{\ddot{\xi}+\ddot{\eta}}^{\ddot{s}}-\lan{\ddot{\eta}}^{\ddot{s}} \right) \\
&\qquad \qquad \times \wt{\cf} \left[\dot{\De}_{\dot{k}} \ddot{S}_{\ddot{k}-3}f\right]\pr{\dot{\xi},\ddot{\xi}} \cf\left[  \dot{S}_{\dot{k}-3} \ddot{\De}_{\ddot{k}}{g}\right]\pr{\dot{\eta},\ddot{\eta}} e^{i\wt{x} \cdot (\wt{\xi}+\wt{\eta})} \, d \wt{\xi} \,d\wt{\eta}.
\end{align*}
Then, for any  $\dot{k},\ddot{k}\in \nn$ and $f,g\in \mathcal S(\rr^{\wt{n}}),$
\begin{align}
\n{T_{1}^{k,k'}\pr{ f,  g} }{\lp}
&\lesssim 2^{\dot{k}(\dot{s}-1) + \ddot{k}(\ddot{s}-1)} \n{\dot{\na} \ddot{\nabla} \dot{S}_{\dot{k}-3} \ddot{S}_{\ddot{k}-3} f}{\lpa} \n{\dot{\De}_{\dot{k}} \ddot{\De}_{\ddot{k}} g}{\lpb},\notag \\
\n{T_{2}^{k,k'}\pr{f,  g} }{\lp}&\lesssim 2^{\dot{k}(\dot{s}-1) + \ddot{k}(\ddot{s}-1)} \n{  \dot{\De}_{\dot{k}} \ddot{\na} \ddot{S}_{\ddot{k}-3} f}{\lpa} \n{\dot{\na} \dot{S}_{\dot{k}-3} \ddot{\De}_{\ddot{k}} g}{\lpb}. \label{5.3b}
\end{align}
\end{prop}

\begin{proof}
Since both estimates can be proved in the similar way, we will only prove \eqref{5.3b} which is slightly more complicated.  Note that the integrand in the definition of $T_2^{k,k'}(f,g)$ is supported on
\[
\left\{\pr{\dot{\xi},\ddot{\xi},\dot{\eta},\ddot{\eta}} : \abs{\dot{\xi}} \sim 2^{\dot{k}}, \,  \abs{\ddot{\xi}} \ll 2^{\ddot{k}}, \, \abs{\dot{\eta}} \ll 2^{\dot{k}}, \,  \abs{\ddot{\eta}} \sim 2^{\ddot{k}}\right\}.
\]

We can write the Fourier symbol of $T_{2}^{k,k'}$ as
\begin{equation}\label{eq:tt}
\pr{\lan{\dot{\xi}+\dot{\eta}}^{\dot{s}}- \lan{\dot{\xi}}^{\dot{s}}} \left( \lan{\ddot{\xi}+\ddot{\eta}}^{\ddot{s}}-\lan{\ddot{\eta}}^{\ddot{s}} \right)
 = ss'\int_0^1 \dot{\eta} \cdot  \pr{ \dot{\xi}+\dot{t} \dot{\eta}} \lan{ \dot{\xi}+\dot{t} \dot{\eta}}^{\dot{s}-2}\, d\dot{t}  \int_0^1 \ddot{\xi} \cdot\pr{\ddot{t}\ddot{\xi}+\ddot{\eta}}\lan{\ddot{t}\ddot{\xi}+\ddot{\eta}}^{s-2} \, d\ddot{t}.
\end{equation}
Noticing that, within the support of the integrand for $T_2^{k,k'}(f,g)$ and for $\dot{t}, \ddot{t} \in [0,1]$, we have
\[
\lan{ \dot{\xi} + \dot{t}\dot{\eta}} \sim  2^{\dot{k}}\quad \textnormal{and} \quad \lan{\ddot{t}\ddot{\xi} + \ddot{\eta}}\sim  2^{\ddot{k}}.
\]
Hence, we may multiply the integrand by  $\dot{\psi}\pr{2^{-\dot{k}} ( \dot{\xi} + \dot{t}\dot{\eta})}$ and $\ddot{\psi}\pr{2^{-\ddot{k}} ( \ddot{\xi} + \ddot{t}\ddot{\eta})}$ for appropriate smooth functions $\dot{\psi}, \ddot{\psi}$ supported on annuli in $\rr^{\dot{n}}$ and $\rr^{\ddot{n}}$, respectively, without altering the value of integral.  For instance, we can rewrite the $d\dot{t}$ integral on the RHS of \eqref{eq:tt} as
\[
\int_0^1 \dot{\eta} \cdot  \left[\pr{ \dot{\xi}+\dot{t} \dot{\eta}} \lan{ \dot{\xi}+\dot{t} \dot{\eta}}^{\dot{s}-2}\dot{\psi}\pr{2^{-\dot{k}} ( \dot{\xi} + \dot{t}\dot{\eta})}\right]\, d\dot{t}= 2^{\dot{k}\pr{\dot{s}-1}} \int_0^1 \dot{\eta} \cdot  \dot{\mathbf{h}}_{\dot{k}}\pr{2^{-\dot{k}} ( \dot{\xi} + \dot{t}\dot{\eta})}\, d\dot{t},
\]
where
\[
\dot{\mathbf{h}}_{\dot{k}}\pr{\dot{\xi}} := \dot{\xi} \pr{2^{-2\dot{k}} + \abs{\dot{\xi}}^2}^{\f{\dot{s}}{2}-1}\dot{\psi}\pr{\dot{\xi}}.
\]
Since $\dot{\psi}$ is supported on an annulus on $\rr^{\dot{n}}$ (in particular, away from the origin), $\{\dot{\mathbf{h}}_{\dot{k}}\}_{\dot{k}\in \nn}$ is a $C^{\infty}$ family of functions supported on some ball $B_{\dot{R}}\pr{\rr^{\dot{n}}}$.  In particular, its derivatives of any order are uniformly bounded in $k\in \nn$.  Thus these $\rr^{\dot{n}}$-valued functions can be expanded into Fourier series whose coefficients, denoted by $\mathbf{c}_{\mm}^{k}$, decay rapidly with constants independent of $\dot{k}\in \nn$;  that is, for any $\dot{M}\gg 1$,
\[
\abs{\dot{\mathbf{c}}_{\dot{\mm}}^{\dot{k}}} \lesssim_{\dot{M}} \lan{\dot{\mm}}^{-\dot{M}}, \quad \forall \dot{\mm}\in \zz^{\dot{n}},
\]
where the implicit constant is independent of $\dot{k}$.  We can treat the $d\ddot{t}$ integral on the RHS of \eqref{eq:tt} similarly where the corresponding Fourier coefficients, denoted as $\mathbf{c}_{\mm'}^{k'}$, have decay rate $\lan{\ddot{\mm}}^{-\ddot{M}}$ for any $\ddot{M}\gg 1$.  Then, the RHS of \eqref{eq:tt} is written as  (up to a constant)
\[
  2^{\dot{k}\pr{\dot{s}-1}} 2^{\ddot{k}\pr{\ddot{s}-1}}\sum_{\tiny \begin{array}{c}\dot{\mm}\in \zz^{\dot{n}}\\ \ddot{\mm}\in \zz^{\ddot{n}}\end{array}}\pr{\int_0^1 \dot{\eta} \cdot \dot{\mathbf{c}}_{\dot{\mm}}^{\dot{k}}  e^{i2^{-\dot{k}} ( \dot{\xi} + \dot{t}\dot{\eta})\cdot \dot{\mm}}\, d\dot{t}}\cdot \pr{\int_0^1 \ddot{\xi}\cdot  \mathbf{c}_{\ddot{\mm}}^{\ddot{k}}  e^{i 2^{-\ddot{k}} ( \ddot{t}\ddot{\xi} + \ddot{\eta}) \cdot \ddot{\mm}}\, d\ddot{t}}.
\]
Putting this expression back into the integrand, we rewrite $T_2(f,g)$ as follows
\begin{align*}
 2^{\dot{k}\pr{\dot{s}-1}} 2^{\ddot{k}\pr{\ddot{s}-1}} \sum_{\tiny \begin{array}{c}\dot{\mm}\in \zz^{\dot{n}}\\ \ddot{\mm}\in \zz^{\ddot{n}}\end{array}} \dot{\mathbf{c}}_{\dot{\mm}}^{\dot{k}} \cdot &\int_0^1     \dot{\na} \dot{S}_{\dot{k}-3}\ddot{\De}_{\ddot{k}} g \pr{\dot{x} + \f{2\pi}{\dot{R}}2^{-\dot{k}} \dot{t}\dot{\mm},\ddot{x} + \f{2\pi}{\ddot{R}}2^{-\ddot{k}} \ddot{\mm}} \, d\dot{t}\\
 & \times  \mathbf{c}_{\ddot{\mm}}^{\ddot{k}}\cdot \int_0^1     \dot{\De}_{\dot{k}} \ddot{\na} \ddot{S}_{\ddot{k}-3} f\pr{ \dot{x} + \f{2\pi}{\dot{R}}2^{-\dot{k}} \dot{\mm},\ddot{x} + \f{2\pi}{\ddot{R}}2^{-\ddot{k}} \ddot{t}\ddot{\mm} }  \, d\ddot{t}.
\end{align*}
Now, we can take the $\lp$~norm of this expression and apply \eqref{eq:mixedtri} as well as H\"older's and Minkowski's inequalities to show that $\n{T_2^{k,k'}(f,g)}{\lp}^{p^*}$ is bounded by
\begin{align*}
2^{\dot{k}\pr{\dot{s}-1} + \ddot{k}\pr{\ddot{s}-1}} \sum_{\tiny \begin{array}{c}\dot{\mm}\in \zz^{\dot{n}}\\ \ddot{\mm}\in \zz^{\ddot{n}}\end{array}} \abs{ \dot{\mathbf{c}}_{\dot{\mm}}^{\dot{k}}  \mathbf{c}_{\ddot{\mm}}^{\ddot{k}}}^{p^*}  & \pr{\int_0^1 \n{  \dot{\na} \dot{S}_{\dot{k}-3}\ddot{\De}_{\ddot{k}} g \pr{\dot{x} + \f{2\pi}{\dot{R}}2^{-\dot{k}} \dot{t}\dot{\mm},\ddot{x} + \f{2\pi}{\ddot{R}}2^{-\ddot{k}} \ddot{\mm}} }{\lpb}\, d\dot{t}}^{p^*}\\
&\times \pr{\int_0^1  \n{  \dot{\De}_{\dot{k}} \ddot{\na} \ddot{S}_{\ddot{k}-3} f\pr{ \dot{x} + \f{2\pi}{\dot{R}}2^{-\dot{k}} \dot{\mm},\ddot{x} + \f{2\pi}{\ddot{R}}2^{-\ddot{k}} \ddot{t}\ddot{\mm} }   }{\lpa} \, d\ddot{t}}^{p^*}.
\end{align*}
Translating along $\dot{x}, \ddot{x}$, taking supremum in $\dot{t}, \ddot{t} \in [0,1]$ as well as $\dot{k}, \ddot{k}\in \nn$, and applying the trivial estimate $\lan{x+y}\le \lan{x}\lan{y}$, the  above is bounded by
\[
2^{\dot{k}\pr{\dot{s}-1} + \ddot{k}\pr{\ddot{s}-1}} \sum_{\tiny \begin{array}{c}\dot{\mm}\in \zz^{\dot{n}}\\ \ddot{\mm}\in \zz^{\ddot{n}}\end{array}} \abs{ \dot{\mathbf{c}}_{\dot{\mm}}^{\dot{k}}  \mathbf{c}_{\ddot{\mm}}^{\ddot{k}}}^{p^*} \lan{\dot{\mm}}^{\f{2\dot{a} p^*}{ \dot{p}_1}}\lan{\ddot{\mm}}^{\f{2\ddot{a} p^*}{ \ddot{p}_1}}    \n{  \dot{\na} \dot{S}_{\dot{k}-3}\ddot{\De}_{\ddot{k}} g  }{\lpb}^{p^*}\, \n{  \dot{\De}_{\dot{k}} \ddot{\na} \ddot{S}_{\ddot{k}-3} f   }{\lpa}^{p^*}.
\]
The series above converges due to the rapid decay of the coefficients, leading to the desired estimate.
\end{proof}

The remainder of the proof of Theorem~\ref{main3} follows from the same arguments from \cite{OW}, together with the modified lemmas given in this section.  The details are omitted.

\end{document}